\documentclass[a4paper]{amsart}
\usepackage{graphicx}
\usepackage{amssymb}
\usepackage{amsmath}
\usepackage{amsthm}
\usepackage{amscd}
\usepackage[all,2cell]{xy}
\usepackage{color}
\usepackage{hyperref}

\UseAllTwocells \SilentMatrices
\newtheorem{thm}{Theorem}[section]

\newtheorem{lem}[thm]{Lemma}

\newtheorem{prop}[thm]{Proposition}
\theoremstyle{definition}

\theoremstyle{remark}
\newtheorem{rem}[thm]{\bf Remark}
\numberwithin{equation}{section}

\begin{document}
\title[Homotopy categories, Leavitt path algebras and GP modules]{Homotopy categories, Leavitt path algebras, and Gorenstein projective modules}

\author[Xiao-Wu Chen,  Dong Yang] {Xiao-Wu Chen, Dong Yang}

\thanks{Chen is supported by National Natural Science Foundation of China (No. 11201446), the Fundamental
Research Funds for the Central Universities (WK0010000024) and Program for New Century Excellent Talents
in University. Yang was supported by DFG - SPP Darstellungstheorie KO1281/9-1 and JSPS.}

\date{\today}

\thanks{E-mail: xwchen$\symbol{64}$mail.ustc.edu.cn, dongyang2002@gmail.com}


\maketitle

\dedicatory{}%
\commby{}%

\begin{abstract}
For a finite quiver without sources or sinks, we prove that the homotopy category of acyclic complexes of injective modules over the corresponding finite-dimensional algebra with radical square zero is triangle equivalent to the derived category of the Leavitt path algebra
viewed as a differential graded algebra with trivial differential, which is further triangle equivalent to the stable category of Gorenstein projective modules over the trivial extension algebra of a von Neumann regular algebra by an invertible bimodule. A related, but different, result for the homotopy category of acyclic complexes of projective modules is given. Restricting these equivalences to compact objects, we obtain various descriptions of  the singularity category of a finite-dimensional algebra with radical square zero, which contain previous results.
\end{abstract}

\section{Introduction}

Let $k$ be a field and $Q$ be a finite quiver without sources or sinks. Denote by $kQ$ the path algebra of $Q$ and by $J$ its two-sided ideal generated by arrows. Set $A=kQ/J^2$ to be the factor algebra modulo the ideal $J^2$; it is a finite-dimensional algebra with radical square zero.

We denote by $\mathbf{K}_{\rm ac}(A\mbox{-Inj})$  the homotopy category of acyclic complexes of injective $A$-modules. This category is called the \emph{stable derived category} of $A$ in \cite{Kra}, which is  a compactly generated triangulated category such that the corresponding full subcategory  consisting of compact objects is triangle equivalent to the  \emph{singularity category} $\mathbf{D}_{\rm sg}(A)$ of $A$ in the sense of \cite{Buc, Or04}. Here, we recall that $\mathbf{D}_{\rm sg}(A)$ is by definition the Verdier quotient category of the bounded derived category of finitely generated $A$-modules modulo the full subcategory consisting of perfect complexes. In a similar fashion, the homotopy category  $\mathbf{K}_{\rm ac}(A\mbox{-Proj})$ of acyclic complexes of projective $A$-modules is related to the singularity category of the opposite algebra $A^{\rm op}$ of $A$.

The aim of this paper is to describe these homotopy categories in terms of Leavitt path algebras and Gorenstein projective modules.
Restricting to the full subcategories consisting of compact objects, we obtain various descriptions of the singularity category of $A$, which contain the corresponding results in \cite{Ch11, Sm}.

Let $L(Q)$ be the \emph{Leavitt path algebra} of $Q$ in the sense of \cite{AA05, AMP}, which is naturally $\mathbb{Z}$-graded as $L(Q)=\bigoplus_{n\in \mathbb{Z}} L(Q)^n$. We view it as a differential graded algebra with trivial differential. Denote by $\mathbf{D}(L(Q))$ the \emph{derived category} \cite{Kel} of differential graded $L(Q)$-modules.

We observe that  each component $L(Q)^n$  is naturally an $L(Q)^0$-bimodule. Indeed, the algebra $L(Q)^0$ is von Neumann regular and the $L(Q)^0$-bimodules $L(Q)^1$ and $L(Q)^{-1}$ are both invertible; compare \cite{Sm, Haz}. We consider the corresponding \emph{trivial extension} algebras $\Lambda^+(Q)=L(Q)^0\ltimes L(Q)^1$ and $\Lambda^-(Q)=L(Q)^0\ltimes L(Q)^{-1}$. Indeed, there is an isomorphism $\Lambda^+(Q)\simeq \Lambda^-(Q)^{\rm op}$, induced by the involution of $L(Q)$.

For an algebra $\Lambda$, we denote by $\Lambda\mbox{-GProj}$ the category of \emph{Gorenstein projective} $\Lambda$-modules \cite{EJ95}; it is naturally a Frobenius category whose projective-injective objects are precisely projective $\Lambda$-modules. Then the \emph{stable category} $\Lambda\mbox{-\underline{GProj}}$ modulo projective modules is naturally triangulated due to a general result in \cite{Hap88}.

 Here is our main result. For notation, $Q^{\rm op}$ is  the opposite quiver of $Q$; $L(Q)^{\rm op}$ is the opposite algebra of $L(Q)$, which is also viewed as a differential graded algebra with trivial differential.

\vskip 5pt

\noindent {\bf Main Theorem.}\quad \emph{Let $Q$ be a finite quiver without sources or sinks. Keep the notation as above. Then there are triangle equivalences
\begin{enumerate}
\item $\mathbf{K}_{\rm ac}(kQ/J^2\mbox{-}{\rm Inj}) \stackrel{\sim}\longrightarrow \;\mathbf{D}(L(Q)^{\rm op})  \stackrel{\sim}\longrightarrow \Lambda^{+}(Q)\mbox{-\underline{\rm GProj}}$;
\item $\mathbf{K}_{\rm ac}(kQ/J^2\mbox{-}{\rm Proj}) \stackrel{\sim}\longrightarrow  \mathbf{D}(L(Q^{\rm op})) \stackrel{\sim}\longrightarrow  \Lambda^{-}(Q^{\rm op})\mbox{-\underline{\rm GProj}}.$
\end{enumerate}}

\vskip 5pt

 Here are three remarks on Main Theorem. (1) In general, the homotopy categories $\mathbf{K}_{\rm ac}(A\mbox{-Inj})$ and $\mathbf{K}_{\rm ac}(A\mbox{-Proj})$ are not equivalent for $A=kQ/J^2$. This fact somehow corresponds to that the definition of the Leavitt path algebra is left--right asymmetric; more precisely, for a quiver $Q$, the graded algebras $L(Q)^{\rm op}$ and $L(Q^{\rm op})$ are not isomorphic in general. (2) The triangulated categories in Main Theorem are compactly generated. These equivalences restrict to triangle equivalences on the corresponding full subcategories consisting of compact objects. We obtain that the following categories are triangle equivalent: the singularity category $\mathbf{D}_{\rm sg}(A)$ of $A$, the perfect derived category ${\rm perf}(L(Q)^{\rm op})$ of $L(Q)^{\rm op}$, and the stable category $\Lambda^+(Q)\mbox{-}\underline{\rm Gproj}$ of finitely presented Gorenstein projective $\Lambda^{+}(Q)$-modules,  which is further equivalent to the category of finitely generated graded projective $L(Q)$-modules. This contains the results in \cite{Ch11,Sm} on the singularity category of $A$. We mention that for a quiver $Q$ that consists of a single vertex with loops,  the singularity category $\mathbf{D}_{\rm sg}(kQ/J^2)$  is studied in \cite[Section 10]{AV} via a completely different approach.  (3) The consideration of Gorenstein projective modules in Main Theorem is inspired by recent work in \cite{RZ}, where an explicit description of finitely presented Gorenstein projective modules over certain trivial extension algebras is given.

The paper is devoted to the proof of Main Theorem. In Section 2, we recall some notation and basic facts on differential graded modules and then prove a version of Koszul duality for algebras with radical square zero, which allows us to pass from the homotopy category of complexes of  injective $A$-modules to the derived category of the  path algebra $kQ$, which is graded with all arrows in degree 1 and is viewed as a differential graded algebra with trivial differential. In Section 3, we study universal localization for graded algebras, which in the hereditary case gives rise to a triangle equivalence between the derived category of the universal localization and a certain perpendicular subcategory of the derived category of the algebra that is localized; see Proposition \ref{prop:perpD}. In Section 4, we recall some facts on the Leavitt path algebra $L(Q)$, which is related to $kQ$ via universal localizations and is strongly graded. In Section 5, we show that under certain conditions the stable category of Gorenstein projective modules over a trivial extension algebra is triangle equivalent to the derived category of a strongly graded algebra; see Proposition \ref{prop:loc}. In Section 6, we combine the results obtained in the previous sections and prove Theorems \ref{thm:A} and \ref{thm:B}, that contain Main Theorem. An application showing the relationship between singular equivalences, graded Morita equivalences, and derived equivalences is given in Proposition \ref{prop:equi}.

Throughout the paper, we work over a fixed field $k$.  That is, all algebras are unital algebras over $k$ and all categories and functors are $k$-linear. Modules are by default left modules. We mention that the same results in Sections 3 and 5 hold when $k$ is replaced by an  arbitrary commutative ring.

\section{Koszul duality for algebras with radical square zero}

In this section, we recall Koszul duality for  algebras with radical square zero, which relates the homotopy category of
complexes of injective modules to the derived category of a differential graded algebra with trivial differential. For this purpose, we first recall from \cite{Kel} some notation on differential graded modules.

\subsection{Differential graded modules} Let $A=\bigoplus_{n\in \mathbb{Z}} A^n$ be a $\mathbb{Z}$-graded algebra.  A (left) graded $A$-module is written as $M=\bigoplus_{n\in \mathbb{Z}} M^n$ such that $A^n.M^m\subseteq M^{n+m}$, where the dot ``." denotes the $A$-action on $M$. Nonzero elements $m$ in $M^n$ are said to be homogeneous of degree $n$, and we denote $|m|=n$. We consider $A$ as a graded $A$-module via  the multiplication.

We denote by $A\mbox{-Gr}$ the abelian category of graded $A$-modules with degree-preserving morphisms. To a graded $A$-module $M$, we associate a graded $A$-module $M(1)$ such that $M(1)=M$ as ungraded modules with the grading given by $M(1)^n=M^{n+1}$. We may associate another graded $A$-module $M[1]$ such that $M[1]^n=M^{n+1}$ and the $A$-action ``$_\circ$" on $M[1]$ is defined as $a_\circ m=(-1)^{|a|} a.m$. This gives rise to two automorphisms $(1)\colon A\mbox{-Gr}\rightarrow A\mbox{-Gr}$ and $[1]\colon A\mbox{-Gr}\rightarrow A\mbox{-Gr}$ of categories. We call the functors $(1)$ and $[1]$ the \emph{degree-shift functor} and the \emph{translation functor}, respectively. For each $n\in \mathbb{Z}$, $(n)$ and $[n]$ denote the $n$-th power of $(1)$ and $[1]$, respectively. We observe that there is a natural isomorphism $M(n)\stackrel{\sim}\longrightarrow M[n]$ in $A\mbox{-Gr}$ sending $m\in M^l$ to $(-1)^{ln} m$.

A differential graded algebra (dg algebra for short) is a graded algebra $A$ with a differential $d\colon A\rightarrow A$ of degree 1 such that $d(ab)=d(a)b+(-1)^{|a|}a d(b)$ for homogeneous elements $a, b$ in $A$. Roughly speaking, a dg algebra is a complex equipped with a compatible multiplication.

Let $A$ be a dg algebra. A \emph{ (left) differential graded} $A$-module (dg $A$-module for short) $M$ is a graded $A$-module $M=\bigoplus_{n\in \mathbb{Z}} M^n$ with a differential $d_M\colon M\rightarrow M$ of degree 1 such that $d_M(a.m)=d(a).m+(-1)^{|a|}a.d_M(m)$ for homogeneous elements $a\in A$ and $m\in M$. For a dg $A$-module $M$, denote by $H(M)=\bigoplus_{n\in \mathbb{Z}} H^n(M)$ the graded space with $H^n(M)$ the $n$-th cohomology of $M$.

We consider the base field $k$ as a trivial dg algebra.  Then a dg $k$-module is identified with a complex of vector spaces over $k$.

A morphism between dg $A$-modules is a morphism of graded $A$-modules which commutes with the differentials. We then have the category $\mathbf{C}(A)$ of dg $A$-modules. The translation functor $[1]\colon \mathbf{C}(A)\rightarrow \mathbf{C}(A)$ is defined such that the graded $A$-module structure of $M[1]$ is defined above and the differential is given by $d_{M[1]}=-d_M$. For a morphism $f\colon M\rightarrow N$ in $\mathbf{C}(A)$, its \emph{mapping cone} ${\rm Con}(f)$ is the dg $A$-module such that ${\rm Con}(f)=N\oplus M[1]$ as graded $A$-modules with the differential $d_{{\rm Con}(f)}=\begin{pmatrix} d_N & f \\
                              0 & d_{M[1]}\end{pmatrix}$.

 We denote by $\mathbf{K}(A)$  the homotopy category of dg $A$-modules. It is a triangulated category: its translation functor $[1]$ is induced from $\mathbf{C}(A)$, and each  triangle in $\mathbf{K}(A)$ is isomorphic to $M\stackrel{f}\rightarrow N\stackrel{\binom{1}{0}}\rightarrow {\rm Con}(f)\stackrel{(0, 1)}\rightarrow M[1]$ for some morphism $f$.  We denote by $\mathbf{D}(A)$ the derived category of dg $A$-modules, which is obtained from $\mathbf{C}(A)$ by formally inverting all quasi-isomorphisms. It is also triangulated with translation functor $[1]$ induced from $
\mathbf{C}(A)$.

  A \emph{right differential graded} $A$-module (right dg $A$-module for short) $M$ is a right-graded $A$-module $M=\bigoplus_{n\in \mathbb{Z}} M^n$ with a differential $d_M\colon M\rightarrow M$ of degree 1 such that $d_M(m.a)=d_M(m).a+(-1)^{|m|}m.d(a)$ for homogeneous elements $a\in A$ and $m\in M$.

We denote by $A^{\rm opp}$ the \emph{opposite dg algebra} of $A$: $A^{\rm opp}=A$ as graded spaces with the same differential, and the multiplication ``$\circ$" on $A^{\rm opp}$ is given such that $a\circ b=(-1)^{|a| |b|} ba$.  We identify a right dg $A$-module  $M$ as a left dg $A^{\rm opp}$-module as follows: $a.m=(-1)^{|a| |m|} m.a$ for homogeneous elements $a\in A^{\rm opp}$ and $m\in M$. In particular, we sometimes understand $\mathbf{K}(A^{\rm opp})$ and  $\mathbf{D}(A^{\rm opp})$ as the homotopy category and the derived category of right dg $A$-modules, respectively.

Let $M$ and $N$ be dg $A$-modules. Set ${\rm Hom}_A(M, N)^n={\rm Hom}_{A\mbox{-}{\rm Gr}}(M, N[n])$ for each $n\in \mathbb{Z}$; it  consists of $k$-linear maps $f\colon M\rightarrow N$ which are homogeneous of degree $n$ and satisfy $f(a.m)=(-1)^{n|a|}a.f(m)$ for all homogeneous elements $a\in A$. The graded vector space ${\rm Hom}_A(M, N)=\bigoplus_{n\in \mathbb{Z}} {\rm Hom}_A(M, N)^n$ has a natural differential $d$ such that $d(f)=d_N\circ f-(-1)^{|f|} f\circ d_M$. Furthermore, ${\rm End}_A(M):={\rm Hom}_A(M, M)$ becomes a dg algebra with this differential and the usual composition as multiplication.

For two  dg $A$-modules $M$ and $N$, there is a natural isomorphism
\begin{align}\label{equ:1}
{\rm Hom}_{\mathbf{K}(A)} (M, N[n])\simeq H^n({\rm Hom}_A(M, N))
\end{align}
for each $n\in \mathbb{Z}$. We observe an isomorphism ${\rm Hom}_A(A, N)\stackrel{\sim}\longrightarrow N$ of complexes  sending $f$ to $f(1)$. Then the above isomorphism induces the following isomorphism:
\begin{align}\label{equ:2}
{\rm Hom}_{\mathbf{K}(A)} (A, N[n])\simeq H^n(N).
\end{align}

A dg $A$-module $N$ is \emph{acyclic} if $H(N)=0$. Recall that $\mathbf{D}(A)$ can be realized as the Verdier quotient category of $\mathbf{K}(A)$ by its full subcategory of acyclic dg $A$-modules. A dg $A$-module $P$ is \emph{homotopically projective} provided that ${\rm Hom}_{\mathbf{K}(A)} (P, N)=0$ for any acyclic dg module $N$, which is equivalent to the condition that the canonical functor $\mathbf{K}(A)\rightarrow\mathbf{D}(A)$ induces an isomorphism
\begin{align}\label{equ:3}
{\rm Hom}_{\mathbf{K}(A)} (P, X)\simeq {\rm Hom}_{\mathbf{D}(A)} (P, X)
 \end{align}for any dg $A$-module $X$. For example, the isomorphism (\ref{equ:2}) implies that for each $n$, $A[n]$ is homotopically projective. Consequently, we have the isomorphism
\begin{align}\label{equ:3D}
{\rm Hom}_{\mathbf{D}(A)} (A, N[n])\simeq H^n(N).
\end{align}

For a triangulated category $\mathcal{T}$ and a class $\mathcal{S}$ of objects, we denote by ${\rm thick}\langle \mathcal{S} \rangle$ the smallest thick subcategory of $\mathcal{T}$ containing $\mathcal{S}$. Here, we recall that a \emph{thick} subcategory is by definition a triangulated subcategory that is closed under direct summands. If $\mathcal{T}$ has arbitrary (set-indexed) coproducts, we denote by ${\rm Loc}\langle \mathcal{S} \rangle$ the smallest triangulated subcategory of $\mathcal{T}$ which contains $\mathcal{S}$ and is closed under arbitrary coproducts.  By \cite[Proposition 3.2]{BN}, we have that ${\rm thick}\langle \mathcal{S} \rangle \subseteq {\rm Loc}\langle \mathcal{S} \rangle$. An object $M$  in $\mathcal{T}$ is \emph{compact} if the functor ${\rm Hom}_\mathcal{T}(M, -)$ commutes with arbitrary coproducts. Denote by $\mathcal{T}^c$ the full subcategory consisting of compact objects; it is a thick subcategory.

A triangulated category $\mathcal{T}$ with arbitrary coproducts is said to be \emph{compactly generated}  \cite{Kel, Nee} provided that there exists a set $\mathcal{S}$ of compact objects such that any nonzero object $T$ satisfies that ${\rm Hom}_\mathcal{T}(S, T[n])\neq 0$ for some $S\in \mathcal{S}$ and $n\in \mathbb{Z}$, which  is equivalent to the condition that $\mathcal{T}={\rm Loc}\langle \mathcal{S}\rangle$; compare  \cite[Lemma 3.2]{Nee}; in this case, we have $\mathcal{T}^c={\rm thick}\langle \mathcal{S}\rangle$. If the above set $\mathcal{S}$ consists of  a single object $S$, we call $S$ a \emph{compact generator} of $\mathcal{T}$.

For a dg algebra $A$, the homotopy category $\mathbf{K}(A)$ and then the derived category $\mathbf{D}(A)$ are triangulated categories with arbitrary coproducts; consult \cite[Lemma 1.5]{BN}. For each $n\in \mathbb{Z}$, the isomorphisms (\ref{equ:2}) and (\ref{equ:3D})  imply that $A[n]$ is compact both in $\mathbf{K}(A)$ and $\mathbf{D}(A)$, since $H^n(-)$ commutes with coproducts. Moreover, the isomorphism (\ref{equ:3D}) implies that $\mathbf{D}(A)$ is compactly generated with $A$ its compact generator; in particular, $\mathbf{D}(A)={\rm Loc}\langle A\rangle$ and $\mathbf{D}(A)^c={\rm thick}\langle A \rangle $;  see \cite[Section 5]{Kel}.  Here, we recall that a dg $A$-module $N$ is zero in $\mathbf{D}(A)$ if and only if $H(N)=0$.

The triangulated subcategory $\mathbf{D}(A)^c={\rm thick}\langle A \rangle$ of $\mathbf{D}(A)$ is called the \emph{perfect derived category }of $A$, and is denoted by ${\rm perf}(A)$.

We view an ordinary algebra $\Lambda$ as a dg algebra concentrated at degree 0. In this case, a dg $\Lambda$-module is just a complex of $\Lambda$-modules. We denote by $\Lambda\mbox{-Mod}$ the abelian category of $\Lambda$-modules. Then $\mathbf{K}(\Lambda)$ coincides with the homotopy category $\mathbf{K}(\Lambda\mbox{-Mod})$ of complexes of $\Lambda$-modules, and $\mathbf{D}(\Lambda)$ coincides with the derived category $\mathbf{D}(\Lambda\mbox{-Mod})$ of complexes of $\Lambda$-modules.

A graded algebra $A$ is viewed as a dg algebra with trivial differential. Then graded $A$-modules are viewed as dg $A$-modules with trivial differentials. On the other hand, for a dg $A$-module $X$, its cohomology $H(X)$ is naturally a graded $A$-module. This gives rise to a cohomological functor   $H\colon \mathbf{D}(A)\rightarrow A\mbox{-Gr}$. We point out that $\mathbf{D}(A)$ and $\mathbf{D}(A\mbox{-Gr})$ are different, but each complex of graded $A$-modules can be viewed as a bicomplex and taking the total complex yields a well-behaved triangle functor $\mathbf{D}(A\mbox{-Gr})\rightarrow\mathbf{D}(A)$; for example, see \cite{KalckYang13}.

\subsection{Differential graded bimodules} Let $B$ be another dg algebra. Recall that
a dg $A$-$B$-bimodule $M$ is a left dg $A$-module as well as a right dg $B$-module such that
$(a.m).b=a.(m.b)$. This implies that the canonical map $A\rightarrow {\rm End}_{B^{\rm opp}}(M)$, sending $a$ to $l_a$ with $l_a(m)=a.m$, is a homomorphism of dg algebras. Similarly, the canonical map $B\rightarrow {\rm End}_A(M)^{\rm opp}$, sending $b$ to $r_b$ with $r_b(m)=(-1)^{|b| |m|} m.b$, is a homomorphism of dg algebras.

A dg $A$-$B$-bimodule $M$ is called \emph{left quasi-balanced} provided that the canonical homomorphism $A\rightarrow {\rm End}_{B^{\rm opp}}(M)$ of dg algebras  is a quasi-isomorphism. Dually, one has the notion of \emph{right quasi-balanced bimodule}.

 Let $M$ be a dg $A$-$B$-bimodule. Consider $DM={\rm Hom}_k(M, k)$ as a dg $k$-module, where $M$ is viewed as a dg $k$-module and $k$ itself is  a dg $k$-module concentrated on degree 0. More explicitly, $(DM)^n={\rm Hom}_k(M^{-n}, k)$ and $d_{DM}(\theta)=(-1)^{n+1}\theta\circ d_M$ for $\theta\in (DM)^n$. Then $DM$ becomes a dg $B$-$A$-bimodule with the actions given by $(b.\theta)(m)=(-1)^{|b|}\theta(m.b)$ and $(\theta.a)(m)=\theta(a.m)$ for homogeneous elements $m\in M$, $a\in A$ and $b\in B$.

A dg $k$-module $M$ is said to be \emph{locally finite}, if each component $M^n$ is finite-dimensional over $k$. The following observation is direct.

\begin{lem}\label{lem:duality}
Let $M$ be a dg $A$-$B$-bimodule which is locally finite. Consider $DM$ as a dg $B$-$A$-bimodule as above. Then $M$ is left quasi-balanced if and only if $DM$ is right quasi-balanced.
\end{lem}

\begin{proof}
There is a canonical homomorphism $D\colon {\rm End}_{B^{\rm opp}}(M)\rightarrow {\rm End}_{B}(DM)^{\rm opp}$ of dg algebras, sending $\theta\in {\rm End}_{B^{\rm opp}}(M)^n$ to $D\theta$. Here, for each $f\in (DM)^l$ we have $(D\theta)(f)=(-1)^{nl}f\circ \theta$, which belongs to $(DM)^{l+n}$. Since $M$ is locally finite, the homomorphism $D$ is an isomorphism. Then the composite $A\rightarrow {\rm End}_{B^{\rm opp}}(M) \stackrel{D}\rightarrow {\rm End}_{B}(DM)^{\rm opp}$ coincides with the canonical homomorphism associated to the dg bimodule $DM$. The desired result follows immediately.
\end{proof}

Let $\mathcal{T}$ be a triangulated category with arbitrary coproducts. An object $X$ is \emph{self-compact} provided that it is a compact object in ${\rm Loc}\langle X\rangle $. The following result is well known; compare \cite[4.3]{Kel} and  \cite[Appendix A]{Kra}. Recall that for a dg $A$-$B$-bimodule $M$ and a dg $A$-module $X$, ${\rm Hom}_A(M, X)$ has a natural structure of dg $B$-module.

\begin{prop}\label{prop:triangle-equi}
Let $M$ be a dg $A$-$B$-bimodule which is right quasi-balanced. Assume that $M$ is self-compact as an object in $\mathbf{K}(A)$. Consider ${\rm Loc}\langle M\rangle$ in $\mathbf{K}(A)$.  Then we have a triangle equivalence
$${\rm Hom}_A(M, -)\colon {\rm Loc}\langle M\rangle \stackrel{\sim}\longrightarrow \mathbf{D}(B).$$
\end{prop}

\begin{proof}
We provide a sketchy proof. Set $F={\rm Hom}_A(M, -)$. The canonical map $B\rightarrow {\rm End}_A(M)^{\rm opp}$, which is a quasi-isomorphism by assumption,  induces an isomorphism  $B\simeq F(M)$ in $\mathbf{D}(B)$. For each $n\in \mathbb{Z}$, we recall the isomorphism  (\ref{equ:3D}) $H^n(B)\simeq {\rm Hom}_{\mathbf{D}(B)}(B, B[n])$ and the isomorphism (\ref{equ:1}) $H^n({\rm Hom}_A(M, M))\simeq {\rm Hom}_{\mathbf{K}(A)}(M, M[n])$. Then the  functor $F$ induces an isomorphism $${\rm Hom}_{\mathbf{K}(A)}(M, M[n])\stackrel{\sim}\longrightarrow  {\rm Hom}_{\mathbf{D}(B)}(B, B[n]).$$ It follows that $F$ induces an equivalence between ${\rm thick}\langle M \rangle $ and ${\rm perf}(B)={\rm thick}\langle B\rangle$; see \cite[Lemma 4.2 (a)]{Kel}. The functor $F$ commutes with arbitrary coproducts since $M$ is self-compact. It follows from \cite[Lemma 4.2 (b)]{Kel} that $F$ is fully faithful. Since $\mathbf{D}(B)={\rm Loc}\langle B\rangle$, we infer that $F$ is dense and thus an equivalence.
\end{proof}

\begin{rem}\label{rem:1}
The above functor $F={\rm Hom}_A(M, -)$ sends $M$ to $B$. Moreover, for any element $b\in Z^0(B)$, the zeroth cocycle of $B$, $F$ sends the homomorphism $r_b\colon M\rightarrow M$ to $r_b\colon B\rightarrow B$. Here, $r_b$ denotes the right action of $b$.  This implies that, for any idempotent $e$ of $Z^0(B)$, $F$ sends the dg $A$-module $Me$ to $Be$.
\end{rem}

\subsection{Koszul duality} Recall that a finite quiver $Q=(Q_0, Q_1; s,t\colon Q_1\rightarrow Q_0)$ consists of a finite set $Q_0$ of vertices and a finite set $Q_1$ of arrows, where the two maps $s$ and $t$ assign to each arrow $\alpha$ its starting vertex $s(\alpha)$ and terminating vertex $t(\alpha)$, respectively. A path of length $n$ in $Q$ is a sequence $\alpha_n\cdots \alpha_2\alpha_1$ of arrows such that $s(\alpha_{i+1})=t(\alpha_i)$ for $1\leq i\leq n-1$; we call $s(p)=s(\alpha_1)$ and $t(p)=t(\alpha_n)$ the starting vertex and the terminating vertex of $p$, respectively. Thus paths of length one are just arrows. We associate to a vertex $i$ a trivial path $e_i$ of length zero, and set $s(e_i)=i=t(e_i)$. A vertex $i$ is a sink if there are no arrows starting at $i$ and a source if there are no arrows terminating at $i$.

We denote by $Q_n$ the set of paths of length $n$. Here, we identify a vertex $i$ with the trivial path $e_i$. Recall that the path algebra $kQ=\bigoplus_{n\geq 0}kQ_n$ has a basis given by paths in $Q$ and its multiplication is given by concatenation of paths. More precisely, for two paths $p$ and $q$ with $s(p)=t(q)$, the product $pq$ is the concatenation $pq$; otherwise,  $pq=0$. In particular, we have that $1=\sum_{i\in Q_0}e_i$ is a decomposition of the unit into pairwise orthogonal idempotents, and that $pe_{s(p)}=p=e_{t(p)}p$ for any path $p$.

We denote by $J$ the two-sided ideal of $kQ$ generated by arrows. Consider the quotient algebra $A=kQ/J^2$; it is a finite-dimensional algebra with radical square zero. Indeed, $A=kQ_0\oplus kQ_1$ as $k$-spaces and its Jacobson radical ${\rm rad}(A)=kQ_1$ satisfies ${\rm rad}(A)^2=0$. Elements in $A$ will be written as $(x, y)$ with $x\in kQ_0$ and $y\in kQ_1$.

We consider the following  dg $k$-module $K=\bigoplus_{n\in\mathbb{Z}} K^n$ such that $K^n=0$ for $n\geq 1$ and $K^{-n}=kQ_{n}\oplus kQ_{n+1}$ for $n
\geq 0$, and the differentials $d^{-n}\colon K^{-n}\rightarrow K^{-n+1}$  satisfy $d^{-n}((a, b))=(0, a)$ with $a\in kQ_n$ and $b\in kQ_{n+1}$ for $n\geq 1$. Observe that $H^0(K)=kQ_0$ and $H^n(K)=0$ for $n\neq 0$.

 Then $K$ is a right dg $A$-module, or equivalently, a complex of right $A$-modules. Indeed, each $K^{-n}$ is a right $A$-module such that $(a, b).(x, y)=(ax,bx+ay)$, where the product $ay$ lies in $kQ_{n+1}$. Moreover, $K^{-n}\simeq kQ_n\otimes_{kQ_0} A$ as right $A$-modules, and thus the right $A$-module $K^{-n}$ is projective. Indeed, $K$, as a complex of right $A$-modules, is a minimal projective resolution of the right $A$-module $kQ_0=A/{\rm rad}(A)$.

 The path algebra $kQ$ is graded by the length grading, that is, its homogeneous component $(kQ)^n$ of degree $n$ is $kQ_n$ for $n\geq 0$, and is zero for $n<0$. We consider it as a dg algebra with trivial differential. Set $B=(kQ)^{\rm opp}$,
 the dg opposite algebra of $kQ$.

 For a path $p$ in $Q$, we define a linear map $\delta_p\colon kQ\rightarrow kQ$ as follows: for any path $q$, $\delta_p(q)=q'$ if $q=pq'$ for some path $q'$; otherwise, $\delta_p(q)=0$.  Then $K$ is a left dg $B=(kQ)^{\rm opp}$-module as follows: for a path $p$ of length $l$ and $(a, b)\in K^{-n}$, $p.(a, b)=(-1)^{ln} (\delta_p(a), \delta_p(b))$ if $n\geq l$; otherwise, $p.(a, b)=0$. One checks readily that $K$ is a dg $B$-$A$-bimodule.

 The following fact is standard, which is essentially contained in \cite[Theorem 2.10.1]{BGS}. Recall that $K$ is isomorphic to the Koszul complex of $A$.

 \begin{lem}\label{lem:quasi-balanced}
 The dg $B$-$A$-bimodule $K$ is left quasi-balanced.
 \end{lem}

 \begin{proof}
 Denote by $Z^n$ and $C^n$ the $n$-th cocycle and coboundary of ${\rm End}_{A^{\rm opp}}(K)$. We observe that any element $f\colon K\rightarrow K$ in $C^n$ satisfies that $f(K^{-n})\subseteq kQ_{1}\subseteq K^0$.  Denote by $H({\rm End}_{A^{\rm opp}}(K))$ the cohomology algebra of ${\rm End}_{A^{\rm opp}}(K)$.   Recall that $K$, viewed as a complex of right $A$-modules, is a minimal projective resolution of $kQ_0$. It follows that $ H^n({\rm End}_{A^{\rm opp}}(K)) \simeq {\rm Ext}^n_{A^{\rm opp}}(kQ_0, kQ_0) \simeq {\rm Hom}_{(kQ_0)^{\rm opp}}(kQ_n, kQ_0)$, whose dimension equals the cardinality of $Q_n$ and thus the dimension of $B^n$.

 Consider the canonical map $\rho\colon B\rightarrow {\rm End}_{A^{\rm opp}}(K)$ induced by the left $B$-action. Since $B$ is a dg algebra with trivial differential, we have $\rho(kQ_n)\subseteq Z^n$. Take a nonzero  $x=\sum_{i=1}^m \lambda_i p_i\in B^n$ with nonzero scalars $\lambda_i$ and pairwise distinct paths $p_i$ of length $n$. Consider $(p_1, 0)\in K^{-n}$. Then $\rho(x) ((p_1, 0))=(-1)^n(e_{s(p_1)}, 0)\in K^0$, which is not in $kQ_{1}.$ In particular, $\rho(x)$ does not belong to $C^n$. Taking cohomologies, we get an embedding $B\rightarrow H({\rm End}_{A^{\rm opp}}(K))$ of graded algebras.  By comparing dimensions,  we infer that this embedding is an isomorphism. This proves that $\rho$ is a quasi-isomorphism.
 \end{proof}

Recall that $A\mbox{-Mod}$ denotes the  abelian category of left $A$-modules. Then we identify $\mathbf{K}(A)$ with $\mathbf{K}(A\mbox{-Mod})$. Denote by $A\mbox{-mod}$ the category of finitely generated $A$-modules, and by $A\mbox{-Inj}$ the category of injective $A$-modules.  We consider  $\mathbf{K}(A\mbox{-Inj})$ the homotopy category of complexes of injective $A$-modules, which is a triangulated subcategory of $\mathbf{K}(A)$ that is closed under coproducts. Denote by $\mathbf{D}^b(A\mbox{-mod})$ the bounded derived category of $A\mbox{-mod}$.

 There is a full embedding ${\bf i}\colon \mathbf{D}^b(A\mbox{-mod})\rightarrow \mathbf{K}(A\mbox{-Inj})$ sending a bounded complex $X$ to its homotopically injective resolution ${\bf i}X$. Here, we recall that each bounded complex $X$ of $A$-modules admits a quasi-isomorphism $X\rightarrow {\bf i}X$ with ${\bf i}X$ a bounded-below complex of injective $A$-modules. The homotopy category $\mathbf{K}(A\mbox{-Inj})$ is compactly generated, and the functor ${\bf i}$ induces a triangle equivalence $\mathbf{D}^b(A\mbox{-mod}) \stackrel{\sim}\longrightarrow \mathbf{K}(A\mbox{-Inj})^c$; see \cite[Section 2]{Kra}.

 For each vertex $i$ in $Q$, we define a  right-graded $kQ$-module $T_i$ as follows. If $i$ is a sink in $Q$, we set $T_i=e_i kQ$. Otherwise, consider the monomorphism $\eta_i \colon e_ikQ(-1) \rightarrow \bigoplus_{\{\alpha\in Q_1\; | \; s(\alpha)=i\}} e_{t(\alpha)}kQ$,  which sends a path $p$ that ends at $i$ to $\sum_{\{\alpha\in Q_1\; | \; s(\alpha)=i\}} \alpha p$. Here, $(-1)$ is the inverse of the degree-shift functor $(1)$. Set $T_i$ to be the cokernel of $\eta_i$, that is, we have the following exact sequence of right-graded $kQ$-modules:
  \begin{align}\label{equ:4}
  0 \longrightarrow e_ikQ(-1) \stackrel{\eta_i}\longrightarrow \bigoplus_{\{\alpha\in Q_1\; | \; s(\alpha)=i\}} e_{t(\alpha)}kQ \longrightarrow T_i\longrightarrow 0.
  \end{align}

 We denote by $Q^{\rm op}$ the opposite quiver of $Q$. Then $kQ^{\rm op}$ is the corresponding path algebra. We observe that $kQ^{\rm op}=(kQ)^{\rm op}$. Here, for any algebra $\Lambda$, $\Lambda^{\rm op}$ denotes the usual opposite algebra of $\Lambda$. Recall that $B=(kQ)^{\rm opp}$ is the opposite dg algebra of $kQ$. These is an isomorphism of graded algebras
 \begin{align}\label{equ:graded}
  kQ^{\rm op}\stackrel{\sim}\longrightarrow B,
  \end{align}
  which sends a path $p$ of length $l$ to $(-1)^{\frac{l(l+1)}{2}} p$.

 We identify right $kQ$-modules as $kQ^{\rm op}$-modules. We mention that the graded $kQ^{\rm op}$-module $T_i$ is indecomposable, since ${\rm End}_{kQ^{\rm op}\mbox{-}{\rm Gr}}(T_i)\simeq k$.

 For each vertex $i$, we denote by $G_i$ the corresponding  simple graded $kQ^{\rm op}$-module that concentrated on degree 0. In other words, we have an exact sequence of graded modules
  \begin{align}\label{equ:5}
  0 \longrightarrow \bigoplus_{\{\alpha\in Q_1\; | \; t(\alpha)=i\}}e_{s(\alpha)}kQ(-1) \stackrel{\xi_i}\longrightarrow  e_ikQ \longrightarrow G_i\longrightarrow 0,
  \end{align}
 where $\xi_i$ sends a path $p$ in $e_{s(\alpha)} kQ$ to $\alpha p$. We note that if $i$ is a source then $\xi_i=0$.

For each vertex $i$ of $Q$, we denote by $S_i$ the corresponding simple $kQ/J^2$-module, and by $P_i$ and $I_i$ its projective cover and injective hull, respectively. We identify modules as stalk complexes concentrated at degree 0.

We have the following version of Koszul duality; compare \cite[Theorem 2.12.1]{BGS}, \cite[Section 10]{Kel} and \cite[Example 5.7]{Kra}.

\begin{thm}\label{prop:Koszul}
Let $Q$ be a finite quiver. Then we have a triangle equivalence
$$ \mathbf{K}(kQ/J^2\mbox{-{\rm Inj}})\stackrel{\sim}\longrightarrow \mathbf{D}(kQ^{\rm op}),$$
which sends ${\bf i} S_i$ to $kQ^{\rm op}e_i$,  ${\bf i}P_i$ to $T_i$,  and $I_i$ to $G_i$ for each vertex $i$.
\end{thm}

We restrict the above equivalence to the subcategories consisting  of compact objects and obtain a triangle equivalence $\mathbf{D}^b(kQ/J^2\mbox{-mod})\simeq {\rm perf}(kQ^{\rm op})$, which is well known; it can be obtained, for example, by applying Lemma~\ref{lem:quasi-balanced} and \cite[Lemma 10.3, The 'exterior' case]{Kel}.

\begin{proof}
Recall that $B=(kQ)^{\rm opp}$ is isomorphic to $kQ^{\rm op}$; see (\ref{equ:graded}).  In particular, we view $T_i$ and $G_i$ as graded $B$-modules. Set $A=kQ/J^2$.

Recall the dg $B$-$A$-bimodule $K$ as above, and then the dual $M=DK$ is naturally a dg $A$-$B$-bimodule; moreover, since $K$ is a projective resolution of the right $A$-module $kQ_0$, $M$ is an injective resolution of $kQ_0$ as a left $A$-module, that is, $M\simeq {\bf i} kQ_0$ in $\mathbf{K}(A)$. It follows from \cite[Proposition 2.3]{Kra} that $M$ is self-compact in $\mathbf{K}(A)=\mathbf{K}(A\mbox{-Mod})$ and ${\rm Loc}\langle M\rangle=\mathbf{K}(A\mbox{-{\rm Inj}})$. Here, we use implicitly the fact that $\mathbf{D}^b(A\mbox{-mod})={\rm thick}\langle kQ_0\rangle$.

By Lemmas \ref{lem:duality} and  \ref{lem:quasi-balanced}, the dg $A$-$B$-bimodule $M$ is right quasi-balanced. Then the desired triangle equivalence follows from Proposition \ref{prop:triangle-equi}. More precisely, we obtain a triangle equivalence $F={\rm Hom}_A(M, -)\colon \mathbf{K}(A\mbox{-Inj})\stackrel{\sim}\longrightarrow \mathbf{D}(B)$.

We observe that $Me_i$ is an injective resolution of $S_i$, that is, ${\bf i} S_i\simeq Me_i$. By Remark \ref{rem:1} we infer $F({\bf i} S_i)\simeq F(Me_i)\simeq Be_i$.

If $i$ is a sink, $P_i\simeq S_i$ and $T_i=kQ^{\rm op}e_i$. Then we already have $F({\bf i}P_i)=T_i$. If $i$ is not a sink, the exact sequence (\ref{equ:4}) of graded $B$-modules induces a triangle in $\mathbf{D}(B)$
$$Be_i[-1]\longrightarrow \bigoplus_{\{\alpha\in Q_1\; | \; s(\alpha)=i\}} Be_{t(\alpha)} \longrightarrow T_i\longrightarrow Be_i.$$
Here, we recall that $Be_i[-1]\simeq Be_i(-1)$; see Subsection 2.1. Applying a quasi-inverse $F^{-1}$ of $F$, we have a triangle in $\mathbf{K}(A\mbox{-{\rm Inj}})$
$$\mathbf{i} S_i[-1] \longrightarrow  \bigoplus_{\{\alpha\in Q_1\; | \; s(\alpha)=i\}} \mathbf{i}S_{t(\alpha)} \longrightarrow F^{-1}(T_i)\longrightarrow \mathbf{i} S_i. $$

Recall the fully faithful triangle functor ${\bf i}\colon \mathbf{D}^b(A\mbox{-mod})\rightarrow \mathbf{K}(A\mbox{-Inj})$. It follows that
there exists a bounded complex $X$ of $A$-modules such that $F^{-1}(T_i)={\bf i}X$ with  a triangle in $\mathbf{D}^b(A\mbox{-mod})$
$$S_i[-1]\longrightarrow \bigoplus_{\{\alpha\in Q_1\; | \; s(\alpha)=i\}} S_{t(\alpha)}\longrightarrow X \longrightarrow S_i.$$
This triangle implies that $X$ is concentrated on degree 0, that is, $X$ is isomorphic to its zeroth cohomology $H^0(X)$; moreover, there is an exact sequence of $A$-modules
$$0\longrightarrow \bigoplus_{\{\alpha\in Q_1\; | \; s(\alpha)=i\}} S_{t(\alpha)}\longrightarrow H^0(X) \longrightarrow S_i\longrightarrow 0.$$
 Recall that the graded $B$-module $T_i$ is indecomposable,  and thus $T_i$ is an indecomposable object in $\mathbf{D}(B)$. It follows that $H^0(X)$ is an indecomposable $A$-module whose top is $S_i$. So we have an epimorphism $P_i\rightarrow H^0(X)$ of $A$-modules; it is an isomorphism by comparing dimensions. Here, we use the fact that the dimension of $P_i$ equals the number of arrows starting at $i$ plus one.  We have  that $F^{-1}(T_i)\simeq {\bf i} X\simeq {\bf i} H^0(X)\simeq {\bf i} P_i$. This implies that $F({\bf i} P_i)=T_i$.

The same argument yields $F(I_i)=G_i$. We apply $F^{-1}$ to the triangle induced by the exact sequence (\ref{equ:5}), and then by a dimension argument we have $F^{-1}(G_i)=I_i$. We omit the details. \end{proof}

\section{Universal localization of graded algebras}

In this section, we study universal localizations of graded algebras and  prove a triangle equivalence induced by a universal localization of a graded-hereditary algebra.

Let $A=\bigoplus_{n\in \mathbb{Z}} A^n$ be a graded algebra. We consider the abelian category $A\mbox{-Gr}$ of graded  $A$-modules.
Recall that for an integer $d$, there is the degree-shift functor $(d)\colon A\mbox{-Gr}\rightarrow A\mbox{-Gr}$.

 Let $\Sigma=\{\xi_i\colon P_i\rightarrow Q_i\}_{i\in I}$ be a family of morphisms between finitely generated graded projective $A$-modules. We denote by $\Sigma^\perp$ the full subcategory of $A\mbox{-Gr}$ consisting of modules $M$ such that ${\rm Hom}_{A\mbox{-}{\rm Gr}} (\xi_i, M(d))$ are isomorphisms for all $i\in I$ and $d\in \mathbb{Z}$.

We recall from \cite{Sch} the notion of universal localization of algebras. Let $R$ be an algebra and $\sigma$ be a family of morphisms between finitely generated projective $A$-modules. A homomorphism $\theta\colon R\rightarrow S$ of algebras is called $\sigma$-inverting if for each morphism $\xi \in \sigma$, the morphism $S\otimes_R \xi$ in $S\mbox{-Mod}$ is invertible; $\theta$ is called a universal localization with respect to $\sigma$ if in addition each $\sigma$-inverting homomorphism $\theta'\colon R\rightarrow S'$ factors uniquely through $\theta$.

We have the graded version of the above notion. Let $B=\bigoplus_{n\in \mathbb{Z}} B^n$ be another graded algebra. A homomorphism $\theta\colon A\rightarrow B$ of graded algebras is \emph{$\Sigma$-inverting} if for all $\xi\in \Sigma$, the morphism $B\otimes_A \xi$  in $B\mbox{-Gr}$ is invertible; $\theta$ is called a \emph{graded universal localization} with respect to $\Sigma$ if in addition every $\Sigma$-inverting homomorphism $\theta'\colon A\rightarrow B'$ of graded algebras factors uniquely through $\theta$.

For a graded algebra $A$, we denote by $U\colon A\mbox{-Gr}\rightarrow A\mbox{-Mod}$ the forgetful functor. For a category $\mathcal{C}$, denote by ${\rm Fun}(\mathcal{C})$ the category of contravariant functors from $\mathcal{C}$ to $k\mbox{-Mod}$, and by ${\rm Mod}(\mathcal{C})$ the full subcategory consisting of $k$-linear functors when $\mathcal{C}$ is $k$-linear. A graded algebra $A$ is \emph{left graded-hereditary} if ${\rm Ext}^2_{A\mbox{-}{\rm Gr}}(-, -)$ vanishes.

  The following result is standard; compare \cite[Chapter 4]{Sch}.

  \begin{prop}\label{prop:gradeduniloc}
  Let $A$ and $\Sigma$ be as above. Then a graded universal localization $\theta_\Sigma\colon A\rightarrow A_\Sigma$ of $A$ with respect to $\Sigma$ exists, and is unique up to isomorphism. Moreover, the following statements hold:
  \begin{enumerate}
  \item the underlying homomorphism $\theta_\Sigma\colon A\rightarrow A_\Sigma$ of ungraded algebras is a universal localization of $A$ with respect to $\sigma=U(\Sigma)$;
  \item the functor of restricting  scalars $\theta_\Sigma^*\colon A_\Sigma\mbox{-{\rm Gr}} \rightarrow A\mbox{-{\rm Gr}}$ induces an equivalence $$A_\Sigma \mbox{-{\rm Gr}}\stackrel{\sim}\longrightarrow \Sigma^\perp;$$
      \item for any graded $A_\Sigma$-modules $M, N$, we have natural isomorphisms \\ ${\rm Hom}_{A\mbox{-}{\rm Gr}}(M, N)\simeq {\rm Hom}_{A_\Sigma\mbox{-}{\rm Gr}}(M, N)$ and ${\rm Ext}^1_{A\mbox{-}{\rm Gr}} (M, N)\simeq {\rm Ext}^1_{A_\Sigma\mbox{-}{\rm Gr}}(M, N)$;
          \item if $A$ is left graded-hereditary, so is $A_\Sigma$.
  \end{enumerate}
  \end{prop}

  \begin{proof}
  The uniqueness of $\theta_\Sigma$ is trivial from the definition. The existence follows from the argument in \cite[Theorem 4.1]{Sch} adapted for categories with automorphisms, which we are going to sketch.

  The category $\mathcal{P}=A\mbox{-grproj}$ of finitely generated graded projective $A$-modules carries an automorphism $(1)$,  the degree-shift functor. Denote by $\mathcal{P}[\Sigma'^{-1}]$ the Gabriel--Zisman localization \cite{GZ67} of $\mathcal{P}$ with respect to  $\Sigma'=\{\xi (d)\; |\; \xi\in \Sigma, d\in \mathbb{Z}\}$. Then $\mathcal{P}[\Sigma'^{-1}]$ has a natural automorphism and the canonical functor $q\colon \mathcal{P}\rightarrow \mathcal{P}[\Sigma'^{-1}]$ commutes with the corresponding automorphisms. Take $\mathcal{P}[\Sigma'^{-1}]^{\rm lin}$ to be  the $k$-linearization of $\mathcal{P}[\Sigma'^{-1}]$: the objects of $\mathcal{P}[\Sigma'^{-1}]^{\rm lin}$ and $\mathcal{P}[\Sigma'^{-1}]$ are the same; for two objects $X, Y$, ${\rm Hom}_{{\mathcal{P}[\Sigma'^{-1}]}^{\rm lin}}(X, Y)$ is the  $k$-vector space with basis ${\rm Hom}_{\mathcal{P}[\Sigma'^{-1}]}(X, Y)$ and the composition of $\mathcal{P}[\Sigma'^{-1}]^{\rm lin}$ is induced from $\mathcal{P}[\Sigma'^{-1}]$. Moreover, $\mathcal{P}[\Sigma'^{-1}]^{\rm lin}$ inherits an automorphism from $\mathcal{P}[\Sigma'^{-1}]$, and the  canonical functor ${\rm can}\colon \mathcal{P}[\Sigma'^{-1}]\rightarrow \mathcal{P}[\Sigma'^{-1}]^{\rm lin}$ commutes with these automorphisms.

  We consider the composite $F\colon\mathcal{P}\stackrel{q}\rightarrow \mathcal{P}[\Sigma'^{-1}]\stackrel{\rm can}\rightarrow \mathcal{P}[\Sigma'^{-1}]^{\rm lin}$. This functor  is not necessarily $k$-linear. Denote by $\mathcal{I}$ the two-sided ideal of $\mathcal{P}[\Sigma'^{-1}]^{\rm lin}$ generated by $F(f+g)-F(f)-F(g)$ and $F(\lambda f)-\lambda F(f)$ for $\lambda\in k$, $f, g\in {\rm Hom}_\mathcal{P}(P, Q)$ and $P, Q\in \mathcal{P}$. We have the factor category $\mathcal{P}_\Sigma:=\mathcal{P}[\Sigma'^{-1}]^{\rm lin}/\mathcal{I}$; it is a $k$-linear category with an automorphism which is also denoted by $(1)$.  Denote by $\pi\colon \mathcal{P}[\Sigma'^{-1}]^{\rm lin}\rightarrow \mathcal{P}_\Sigma$ the canonical functor. Then $\pi\circ F\colon \mathcal{P}\rightarrow \mathcal{P}_\Sigma$ is $k$-linear which acts on objects as the identity. In particular, every object in $\mathcal{P}_\Sigma$ is a direct summand of an object $\bigoplus_{i=1}^n A(d_i)$ for some $n\geq 1$ and $d_i\in \mathbb{Z}$.

  Consider the graded algebra $\bigoplus_{n \in \mathbb{Z}} {\rm Hom}_{\mathcal{P}_\Sigma}(A, A(n))$ and define $A_\Sigma$ to be its opposite algebra. The functor $\pi\circ F\colon \mathcal{P}\rightarrow \mathcal{P}_\Sigma$ induces a homomorphism $\theta_\Sigma\colon A\rightarrow A_\Sigma$  of graded algebras, which is universal with respect to the $\Sigma$-inverting property.

   We mention that the idempotent completion  $\widetilde{\mathcal{P}_\Sigma}$ of $\mathcal{P}_\Sigma$, as categories with automorphisms, is equivalent to $A_\Sigma\mbox{-grproj}$, the category of finitely generated projective $A_\Sigma$-modules. Then the composite functor $\mathcal{P}\stackrel{\pi\circ F}\rightarrow \mathcal{P}_\Sigma\rightarrow \widetilde{\mathcal{P}_\Sigma}$ is identified with the tensor functor $A_\Sigma\otimes_A-\colon A\mbox{-grproj}\rightarrow A_\Sigma\mbox{-grproj}$. Here, $\mathcal{P}_\Sigma\rightarrow \widetilde{\mathcal{P}_\Sigma}$ is the canonical embedding of the idempotent completion.

For (1), we recall that for an ungraded algebra $R$ the Laurent polynomial algebra $R[x, x^{-1}]$ is graded by ${\rm deg}\; R=0$, ${\rm deg}\; x=1$ and ${\rm deg}\; x^{-1}=-1$. There is a bijection from the set of algebra homomorphisms from $A$ to $R$  to  the set of graded algebra homomorphisms from $A$ to $R[x, x^{-1}]$, sending $\phi\colon A\rightarrow R$ to $\phi^h\colon A\rightarrow R[x, x^{-1}]$ such that $\phi^h(a)=\phi(a)x^n$ for any $a\in A^n$; moreover, $\phi$ is $\sigma$-inverting if and only if $\phi^h$ is $\Sigma$-inverting. Here, we use the equivalence $R[x, x^{-1}]\mbox{-Gr}\stackrel{\sim}\longrightarrow R\mbox{-Mod}$ of categories, which sends a graded module $M$ to its zeroth component; compare \cite[Theorem 2.8]{Dade}. This bijection implies that $\theta_\Sigma$ is a universal localization with respect to $\sigma$.

For (2), we recall the equivalence $A\mbox{-Gr} \stackrel{\sim}\longrightarrow {\rm Mod}(\mathcal{P})$ which sends a graded $A$-module $M$ to the functor ${\rm Hom}_{A\mbox{-}{\rm Gr}}(-, M)$ restricted on $\mathcal{P}$. It follows that $\Sigma^\perp$ is equivalent to the category of $k$-linear functors which invert morphisms in $\Sigma'$. Recall from \cite[Lemma I.1.2]{GZ67} that ${\rm Fun}(\mathcal{P}[\Sigma'^{-1}])$ identifies with the subcategory of ${\rm Fun}(\mathcal{P})$ consisting of functors which invert morphisms in $\Sigma'$. Hence, $\Sigma^\perp$ is equivalent to the subcategory of ${\rm Fun}(\mathcal{P}[\Sigma'^{-1}])$ consisting of contravariant functors $G\colon \mathcal{P}[\Sigma'^{-1}]\rightarrow k\mbox{-Mod}$ such that $G(q(f+g))=G(q(f)+q(g))$ and $G(q(\lambda f))=\lambda G(q(f))$ for all $f, g\in {\rm Hom}_\mathcal{P}(P, Q)$ and $\lambda \in k$.  Here, we recall the canonical functor $q\colon \mathcal{P}\rightarrow \mathcal{P}[\Sigma'^{-1}]$.

We observe the obvious equivalence ${\rm Fun}(\mathcal{P}[\Sigma'^{-1}])\simeq {\rm Mod}(\mathcal{P}[\Sigma'^{-1}]^{\rm lin})$. Then $\Sigma^\perp$ is further equivalent to the subcategory of ${\rm Mod}(\mathcal{P}[\Sigma'^{-1}]^{\rm lin})$ consisting of functors $G$ which vanish on $\mathcal{I}$, while the latter category is equivalent to ${\rm Mod}(\mathcal{P}_\Sigma)$. Recall from the above the equivalence $\widetilde{\mathcal{P}_\Sigma}\simeq A_\Sigma\mbox{-grproj}$ and from \cite[Proposition 1.3]{BS} the equivalence ${\rm Mod}(\mathcal{P}_\Sigma)\simeq {\rm Mod}(\widetilde{\mathcal{P}_\Sigma})$. Combining these equivalences, we infer that $\Sigma^\perp$ is equivalent to ${\rm Mod} (A_\Sigma\mbox{-grproj})$. Finally, we observe the equivalence  $A_\Sigma\mbox{-Gr}\stackrel{\sim}\longrightarrow {\rm Mod} (A_\Sigma\mbox{-grproj})$. This yields the required equivalence in (2).

Observe that $\Sigma^\perp$ is closed under extensions in $A\mbox{-Gr}$. Then  the two isomorphisms in (3) follow from (2). Recall that a graded algebra $B$ is left graded-hereditary if and only if the functor ${\rm Ext}^1_{B\mbox{-}{\rm Gr}}(M, -)$ is right exact for each graded $B$-module $M$; compare \cite[Lemma A.1]{RV}. Then (4) follows from (3) immediately.
\end{proof}

We view a graded algebra $A$ as a dg algebra with trivial differential, and thus graded $A$-modules as dg $A$-modules with trivial differentials. The following result characterizes certain Hom spaces in the derived category $\mathbf{D}(A)$ of dg $A$-modules; see \cite{KalckYang13} for a more general result.

\begin{lem}\label{lem:hom}
Let $A$ be a graded algebra. Then for any graded $A$-modules $M$ and $N$, there is a natural isomorphism
$${\rm Hom}_{\mathbf{D}(A)} (M, N)\simeq \prod_{i\geq 0} {\rm Ext}^i_{A\mbox{-}{\rm Gr}}(M, N(-i)).$$
\end{lem}

\begin{proof}
Take a projective resolution
$$\cdots \rightarrow P^{-2}\stackrel{d^{-2}}\rightarrow P^{-1}\stackrel{d^{-1}}\rightarrow P^0\stackrel{\varepsilon}\rightarrow M\rightarrow 0$$
 in $A\mbox{-Gr}$, where each $P^{-i}=\bigoplus_{j\in \mathbb{Z}} P^{-i, j}$ is a graded projective $A$-module. Denote this resolution by $P$. Then ${\rm Ext}^i_{A\mbox{-}{\rm Gr}}(M, N(-i))\simeq H^i({\rm Hom}_{A\mbox{-}{\rm Gr}} (P, N(-i)))$ for each $i\geq 0$.

 We associate a dg $A$-module $T$ to the projective resolution $P$.  As a graded $A$-module $T=\bigoplus_{i\geq 0} P^{-i}(i)$ and the restriction of the differential $d_T$ on  $P^{-i, j}$ is given by $(-1)^j d^{-i}$, where we set $d^0=0$. There is a quasi-isomorphism $\epsilon\colon T\rightarrow M$ of dg $A$-modules, whose restriction on $P^0$ is $\varepsilon$ and on $P^{-i}$ is zero for each $i\geq 1$.

 Set $T_n=\bigoplus_{i=0}^n P^{-i}(i)$ for each $n\geq 0$. Then
 $$T_0\subseteq T_1\subseteq T_2\subseteq \cdots$$
  is a sequence of dg submodules of $T$ satisfying that $\bigcup_{n\geq 0}T_n=T$ and $T_{n}/{T_{n-1}}\simeq P^{-n}(n)$ as dg $A$-modules. This implies that the dg $A$-module $T$ is homotopically projective. In particular, the quasi-isomorphism $\epsilon \colon T\rightarrow M$ is a homotopically projective resolution; see \cite[8.1.1]{Kel98}. Hence, by (\ref{equ:3}) and (\ref{equ:1}), we have the following isomorphisms:
 $${\rm Hom}_{\mathbf{D}(A)} (M, N)\simeq {\rm Hom}_{\mathbf{K}(A)} (T, N)\simeq H^0({\rm Hom}_A(T, N)).$$ We observe that there is an isomorphism $${\rm Hom}_A(T, N)\simeq \prod_{i\geq 0}{\rm Hom}_{A\mbox{-}{\rm Gr}} (P, N(-i))[i]$$ of dg $k$-modules. Taking the zeroth cohomologies, we obtain the required isomorphism.
\end{proof}

\begin{rem}\label{rem:hom}
We mention that the above isomorphism is compatible with algebra homomorphisms. More precisely, let $\theta\colon A\rightarrow B$ be a homomorphism of graded algebras. Consider the functors of restricting scalars $\theta^*\colon B\mbox{-Gr}\rightarrow A\mbox{-Gr}$ and $\theta^*\colon \mathbf{D}(B)\rightarrow \mathbf{D}(A)$. Let $M$ and $N$ be graded $B$-modules. Then we have the following commutative diagram:
 \[\xymatrix{
 {\rm Hom}_{\mathbf{D}(B)} (M, N) \ar[d]^-{\theta^*} \ar[rr]^-{\phi_{M, N}} && \prod_{i\geq 0} {\rm Ext}^i_{B\mbox{-}{\rm Gr}}(M, N(-i)) \ar[d]^-{\theta^*}\\
 {\rm Hom}_{\mathbf{D}(A)} (\theta^*(M), \theta^*(N))  \ar[rr]^-{\phi_{\theta^*(M), \theta^*(N)}} && \prod_{i\geq 0} {\rm Ext}^i_{A\mbox{-}{\rm Gr}}(\theta^*(M), \theta^*(N)(-i)),
 }\]
 where $\phi_{M, N}$ and $\phi_{\theta^*(M), \theta^*(N)}$ are the isomorphisms in the preceding lemma. The commutativity follows from the  construction of these isomorphisms. \hfill $\square$
\end{rem}

The following result might be deduced from \cite[Theorems 3.1 and 3.6]{KellerYangZhou09}.

\begin{lem}\label{lem:isoH}
Let $A$ be a graded algebra which is left graded-hereditary. Then each dg $A$-module $X$ is isomorphic to its cohomology $H(X)$ in $\mathbf{D}(A)$. Moreover, a dg $A$-module $X$ is perfect if and only if $H(X)$, as a graded $A$-module, is finitely presented.
\end{lem}

\begin{proof}
Denote by $Z$ and $C$ the cocycle and coboundary of $X$; they are graded $A$-modules, or equivalently, dg $A$-modules with trivial differentials. Consider the natural exact sequences $0\rightarrow Z\stackrel{\rm inc}\rightarrow X \stackrel{d}\rightarrow C[1]\rightarrow 0$ and $0\rightarrow C[1]\stackrel{\rm inc}\rightarrow Z[1]\stackrel{\pi}\rightarrow H(X)[1]\rightarrow 0$ in $A\mbox{-Gr}$. Since $A$ is left graded-hereditary, there exists a graded $A$-module $E$ with the following commutative diagram with exact rows and columns in $A\mbox{-Gr}$.
\[\xymatrix{ & & 0\ar[d] & 0\ar[d]\\
0\ar[r] & Z \ar@{=}[d] \ar[r]^-{{\rm inc}} & X\ar[d]^-{f} \ar[r]^-{d} & C[1]\ar[d]^{\rm inc} \ar[r] & 0\\
0\ar[r] & Z \ar[r]^{f\circ {\rm inc}} & E \ar[r]^-{g} \ar[d]^-{\pi\circ g} & Z[1] \ar[r] \ar[d]^-{\pi} & 0\\
& & H(X)[1]\ar[d] \ar@{=}[r] & H(X)[1] \ar[d]\\
& & 0  &0}
\]
Consider the morphism $f[-1]\colon X[-1]\rightarrow E[-1]$ in $A\mbox{-Gr}$ as a morphism of dg $A$-modules, where $X[-1]$ is viewed as a graded $A$-module, that is, a dg module with trivial differential. Recall the mapping cone ${\rm Con}(f[-1])=E[-1]\oplus X$ in Subsection 2.1.  We observe two quasi-isomorphisms $({\rm inc}\circ (g[-1]), {\rm Id}_X)\colon {\rm Con}(f[-1]) \rightarrow X$ and  $((\pi\circ g)[-1], 0)\colon {\rm Con}(f[-1])\rightarrow H(X)$.  This proves that $X$ is isomorphic to $H(X)$ in $\mathbf{D}(A)$.

Recall that ${\rm perf}(A)={\rm thick}\langle A\rangle$ and that $H\colon \mathbf{D}(A)\rightarrow A\mbox{-Gr}$ is a cohomological functor. We observe that $\{X\in \mathbf{D}(A)\; |\; H(X)\in A\mbox{-Gr} \mbox{ is finitely presented}\}$ is a thick subcategory of $\mathbf{D}(A)$ that contains $A$. Then the ``only if" part of the second statement follows. For the ``if" part, assume that $H(X)$ is finitely presented. Then we have an exact sequence $0\rightarrow P_1\rightarrow P_0 \rightarrow H(X)\rightarrow 0$ of graded $A$-modules with $P_i$ finitely generated projective. This sequence gives rise to a triangle in $\mathbf{D}(A)$.  Since each $P_i$ is perfect, it follows that $H(X)$ and thus $X$ are perfect.
\end{proof}

For a class $\mathcal{S}$ of objects in a triangulated category $\mathcal{T}$, consider the \emph{right perpendicular subcategory} $\mathcal{S}^\perp=\{X\in \mathcal{T}\; |\; {\rm Hom}_\mathcal{T}(S, X[n])=0 \mbox{ for all } S\in \mathcal{S}, n\in \mathbb{Z}\}$; it is a thick subcategory of $\mathcal{T}$.

The following is the main result of this section, the ungraded version of which is known, for example, by combining \cite[Section 1.8]{AngeleriKoenigLiu11} and \cite[Theorem 6.1]{KrauseStovicek10}.

\begin{prop}\label{prop:perpD}
Let $A$ be a graded algebra which is left graded-hereditary. Let $\Sigma$ be a family of morphisms between finitely generated graded projective $A$-modules. Denote by $\theta_\Sigma\colon A\rightarrow A_\Sigma$ the corresponding graded universal localization. Consider the right perpendicular subcategory $\{{\rm Ker}\;\xi, {\rm Cok}\; \xi\; |\; \xi \in \Sigma\}^\perp$ in $\mathbf{D}(A)$.  Then we have a triangle equivalence $$\{{\rm Ker}\;\xi, {\rm Cok}\; \xi\; |\; \xi \in \Sigma\}^\perp \stackrel{\sim}\longrightarrow \mathbf{D}(A_\Sigma).$$
\end{prop}

\begin{proof}
Write $\mathcal{N}=\{{\rm Ker}\;\xi, {\rm Cok}\; \xi\; |\; \xi \in \Sigma\}^\perp$. For a morphism $\xi \colon P\rightarrow Q$ in $\Sigma$, we observe that ${\rm Im}\; \xi$ is projective and then ${\rm Ker}\; \xi$  is a direct summand of $P$; moreover, there is a decomposition $P={\rm Ker}\; \xi\oplus P'$ such that $\xi=(0, \xi')$ for some monomorphism $\xi'\colon P'\rightarrow Q$ with ${\rm Cok}\; \xi={\rm Cok}\; \xi'$. In particular, for a graded $A$-module $M$, ${\rm Hom}_{A\mbox{-}{\rm Gr}}(\xi, M)$ is an isomorphism if and only if ${\rm Hom}_{A\mbox{-}{\rm Gr}} ({\rm Ker}\; \xi, M)=0={\rm Hom}_{A\mbox{-}{\rm Gr}}({\rm Cok}\; \xi, M)={\rm Ext}^1_{A\mbox{-}{\rm Gr}}({\rm Cok}\; \xi, M)$.

Take $X\in \mathbf{D}(A)$. Recall from Lemma \ref{lem:isoH} that $X\simeq H(X)$. Then $X$ lies in $\mathcal{N}$ if and only if ${\rm Hom}_{\mathbf{D}(A)} ({\rm Ker}\; \xi, H(X)[n])=0={\rm Hom}_{\mathbf{D}(A)}({\rm Cok}\; \xi, H(X)[n])$ for all $\xi\in \Sigma$ and $n\in \mathbb{Z}$. By Lemma \ref{lem:hom} this is equivalent to ${\rm Hom}_{A\mbox{-}{\rm Gr}} ({\rm Ker}\; \xi, H(X)(n))=0={\rm Hom}_{A\mbox{-}{\rm Gr}}({\rm Cok}\; \xi, H(X)(n))={\rm Ext}^1_{A\mbox{-}{\rm Gr}}({\rm Cok}\; \xi, H(X)(n))$. By above,  these equalities are equivalent to ${\rm Hom}_{A\mbox{-}{\rm Gr}}(\xi, H(X)(n))$ are isomorphisms for all $n\in \mathbb{Z}$. In summary, $X$ lies in $\mathcal{N}$ if and only if $H(X)$, viewed as an object in $A\mbox{-Gr}$, lies in $\Sigma^\perp$, or equivalently, $\mathcal{N}=\{X\in \mathbf{D}(A)\; |\; H(X) \in \Sigma^\perp\}$.

Consider the graded universal localization $\theta_\Sigma\colon A\rightarrow A_\Sigma$ and the functor of restricting scalars $\theta_\Sigma^*\colon \mathbf{D}(A_\Sigma)\rightarrow \mathbf{D}(A)$. For any graded $A_\Sigma$-modules $M, N$, $\theta_\Sigma^*$ induces an isomorphism ${\rm Hom}_{\mathbf{D}(A_\Sigma)}(M, N)\simeq {\rm Hom}_{\mathbf{D}(A)}(\theta_\Sigma^*(M), \theta_\Sigma^*(N))$. This follows from Lemma \ref{lem:hom}, Remark \ref{rem:hom} and Proposition \ref{prop:gradeduniloc}(3). By Proposition \ref{prop:gradeduniloc}(4) $A_\Sigma$ is left graded-hereditary and then by Lemma \ref{lem:isoH} each object in $\mathbf{D}(A_\Sigma)$ is isomorphic to a graded $A_\Sigma$-module. Then we infer that $\theta_\Sigma^*\colon \mathbf{D}(A_\Sigma)\rightarrow \mathbf{D}(A)$ is fully faithful; moreover,  it follows from Proposition \ref{prop:gradeduniloc}(2) that its essential image  is $\{X\in \mathbf{D}(A)\; |\; H(X) \in \Sigma^\perp\}$, which equals $\mathcal{N}$. Then we have the required equivalence.
\end{proof}

\section{Leavitt path algebras}

In this section, we recall some results on Leavitt path algebras \cite{AA05, AMP}, which we will use in the proof of Main Theorem.

Let $Q$ be a finite quiver. Denote by $\bar{Q}$ its double quiver, which is obtained from $Q$ by adding for each $\alpha$ in $Q_1$ an arrow $\alpha^*$ going in the opposite direction; the added arrows are called \emph{ghost arrows}. Denote by $k\bar{Q}$ the path algebra of the double quiver $\bar{Q}$.

 The \emph{Leavitt path algebra} $L(Q)$ of $Q$ is by definition the quotient algebra of $k\bar{Q}$ modulo the ideal generated by $\{\alpha \beta^*-\delta_{\alpha, \beta} e_{t(\alpha)}, \sum_{\{\alpha\in Q_1\; |\; s(\alpha)=i\}} \alpha^*\alpha-e_i\; |\; \alpha, \beta\in Q_1, \; i \in Q_0 \mbox{ that are not sinks}\}$. Here, $\delta$ is the Kronecker symbol. The generators of the ideal are known as  the \emph{Cuntz--Krieger relations}.

The Leavitt path algebra $L(Q)=\bigoplus_{n\in \mathbb{Z}} L(Q)^n$ is naturally $\mathbb{Z}$-graded by means of ${\rm deg}\; e_i=0$, ${\rm deg}\; \alpha=1$  and ${\rm deg}\; \alpha^*=-1$ for all $i\in Q_0$ and $\alpha\in Q_1$. There is a homomorphism of graded algebras $ \iota_Q\colon kQ\rightarrow L(Q)$ such that $\iota_Q(e_i)=e_i$ and $\iota_Q(\alpha)=\alpha$. It is well known that the homomorphism $\iota_Q$ is injective; for example, see \cite[Proposition 4.1]{Ch12}.

There is an algebra anti-automorphism $(-)^*\colon L(Q)\rightarrow L(Q)$, called the \emph{involution}, such that $(e_i)^*=e_i$, $(\alpha)^*=\alpha^*$ and $(\alpha^*)^*=\alpha$. Note that $(-)^*$ identifies $L(Q)^n$ with $L(Q)^{-n}$.

 Recall that a graded algebra $A=\bigoplus_{n\in \mathbb{Z}} A^n$ is \emph{strongly graded} provided that $A^n A^m=A^{n+m}$ for all $n, m\in \mathbb{Z}$, which is equivalent to the conditions that $A^1 A^{-1}=A^0=A^{-1} A^1$. In this case, each $A^0$-bimodule $A^n$ is invertible.

 An algebra $\Lambda$ is \emph{von Neumann regular} provided that all $\Lambda$-modules are flat, or equivalently, any finitely presented $\Lambda$-module is projective. Then the category $\Lambda\mbox{-proj}$ of finitely generated projective $\Lambda$-modules is a semisimple abelian category. Here, we recall that an abelian category is \emph{semisimple} if any short exact sequence splits.

 For a semisimple abelian category $\mathcal{A}$ and an auto-equivalence $\Sigma$ on it, we denote by $(\mathcal{A}, \Sigma)$ the unique triangulated structure on $\mathcal{A}$ with the translation functor $\Sigma$; see \cite[Lemma 3.4]{Ch11}. Indeed, all triangles in this category are direct sums of trivial ones. For example, given an invertible $\Lambda$-bimodule $K$ on a von Neumann regular algebra $\Lambda$, we consider the auto-equivalence $K\otimes_\Lambda-$ on $\Lambda\mbox{-proj}$, and thus we have a triangulated category $(\Lambda\mbox{-proj}, K\otimes_\Lambda-)$.

The following result summarizes the structure of $L(Q)=\bigoplus_{n\in \mathbb{Z}} L(Q)^n$ as a graded algebra. We denote by $L(Q)\mbox{-}{\rm grproj}$ the full subcategory of $L(Q)\mbox{-Gr}$ consisting of finitely generated graded projective $L(Q)$-modules, and recall that $(1)$ denotes the degree-shift functor of $L(Q)$-modules.

\begin{lem}\label{lem:Leavitt}
Let $Q$ be a finite quiver without sinks. Then the following statements hold:
\begin{enumerate}
\item the Leavitt path algebra $L(Q)$ is strongly graded; in particular, we have an equivalence $L(Q)\mbox{-}{\rm grproj}\stackrel{\sim}\longrightarrow L(Q)^0\mbox{-}{\rm proj}$;
\item the zeroth component subalgebra $L(Q)^0$ is von Neumann regular and hereditary;
\item there are triangle equivalences
$${\rm perf}(L(Q))\stackrel{\sim}\longrightarrow (L(Q)\mbox{-}{\rm grproj}, (1)) \stackrel{\sim}\longrightarrow (L(Q)^0\mbox{-}{\rm proj}, L(Q)^1\otimes_{L(Q)^0}-);$$
\item there is a triangle equivalence
$$(L(Q)^{\rm op}\mbox{-}{\rm grproj}, (1))\stackrel{\sim}\longrightarrow (L(Q)\mbox{-}{\rm grproj}, (-1)).$$
\end{enumerate}
\end{lem}

\begin{proof}
For (1), we recall from \cite[5.3]{Sm} or \cite[Theorem 3.5]{Haz} that $L(Q)$ is strongly graded. Then by \cite[Theorem 2.8]{Dade} the functor  $(-)^0\colon L(Q)\mbox{-Gr}\rightarrow L(Q)^0\mbox{-Mod}$ sending a graded module $M=\bigoplus_{n\in \mathbb{Z}} M^n$ to its zeroth component $M^0$ is an equivalence; moreover, via this equivalence the degree-shift functor $(1)$ corresponds to the functor $L(Q)^1\otimes_{L(Q)^0}-$. The required equivalence follows directly.

By the proof of  \cite[Theorem 5.3]{AMP} or \cite[5.3]{Sm}, the algebra $L(Q)^0$ is a direct limit of finite-dimensional semisimple algebras, and thus it is von Neumann regular and hereditary; this yields (2). The statement (3) is a special case of Lemma \ref{lem:perfder}.

The equivalence in (4) follows from the involution $(-)^*\colon L(Q)\rightarrow L(Q)$ , which induces an equivalence $L(Q)^{\rm op}\mbox{-}{\rm grproj}\stackrel{\sim}\longrightarrow L(Q) \mbox{-}{\rm grproj}$. This equivalence identifies $(1)$ on $L(Q)^{\rm op}\mbox{-}{\rm grproj}$ with $(-1)$ on $L(Q)\mbox{-}{\rm grproj}$.
\end{proof}

\begin{lem}\label{lem:perfder}
Let $A=\bigoplus_{n\in \mathbb{Z}} A^n$ be a strongly graded algebra such that $A^0$ is von Neumann regular and left hereditary.
Then we have a triangle equivalence
$${\rm perf}(A)\stackrel{\sim}\longrightarrow (A\mbox{-}{\rm grproj}, (1)) \stackrel{\sim}\longrightarrow (A^0\mbox{-}{\rm proj}, A^1\otimes_{A^0}-).$$
\end{lem}

By the equivalence $A\mbox{-grproj}\stackrel{\sim}\longrightarrow A^0\mbox{-proj}$, the category  $A\mbox{-grproj}$ is semisimple abelian. Thus the triangulated category $(A\mbox{-}{\rm grproj}, (1))$ is defined, and it is equivalent to $(A^0\mbox{-proj}, A^1\otimes_{A^0}-)$. Here, we use implicitly that for a graded $A$-module $M$, there is an isomorphism $M^1\simeq A^1\otimes_{A^0} M^0$ of $A^0$-modules.

\begin{proof}
We infer from the assumption and the equivalence $A\mbox{-Gr} \stackrel{\sim}\longrightarrow A^0\mbox{-Mod}$ that  $A$ is left graded-hereditary and every finitely presented graded $A$-module is projective. We apply Lemma \ref{lem:isoH} to get  ${\rm perf}(A)=\{X\simeq H(X)\in \mathbf{D}(A)\; |\; H(X)\in A\mbox{-}{\rm grproj}\}$, and thus a well-defined functor $H\colon {\rm perf}(A)\rightarrow A\mbox{-}{\rm grproj}$. This functor is  an equivalence by Lemma \ref{lem:hom}. Moreover, it follows from the definition of $H$ that the translation functor $[1]$ on ${\rm perf}(A)$ corresponds to the degree-shift functor $(1)$ on  $A\mbox{-}{\rm grproj}$, and hence $H$ is a triangle equivalence.
\end{proof}

Recall the opposite quiver $Q^{\rm op}$ of $Q$ and the natural identification
 $kQ^{\rm op}=(kQ)^{\rm op}$. Then we have a homomorphism of graded algebras
$$\kappa_Q \colon kQ=(kQ^{\rm op})^{\rm op} \stackrel{(\iota_{Q^{\rm op}})^{\rm op}} \longrightarrow L(Q^{\rm op})^{\rm op}.$$

For a vertex $j$ in $Q$, $kQe_j$ denotes the left homogeneous ideal of $kQ$ generated by $e_j$, where $e_j$ is homogeneous of degree 0. Observe that $kQe_j$ has a basis given by all paths starting at $j$.  We recall the homomorphisms  $\eta^Q_i$ and $\xi^Q_i$ in the sequences (\ref{equ:4}) and (\ref{equ:5}) for left graded $kQ$-modules. More precisely,  if  $i\in Q_0$ is not a sink, there is an exact sequence of left graded $kQ$-modules
\begin{align}\label{equ:xi}
  0 \longrightarrow \bigoplus_{\{\alpha\in Q_1\; | \; s(\alpha)=i\}}kQe_{t(\alpha)}(-1) \stackrel{\xi^Q_i}\longrightarrow  kQe_i \longrightarrow G_i^L\longrightarrow 0,
  \end{align}
  where $\xi^Q_i(p)=p\alpha$ for any path $p$ starting at $t(\alpha)$, and  $G_i^L$ is a graded simple $kQ$-module concentrated on degree 0; if $i\in Q_0$ is not a source, there is an exact sequence of left graded $kQ$-modules
 \begin{align}\label{equ:eta}
  0 \longrightarrow kQe_i(-1) \stackrel{\eta^Q_i}\longrightarrow \bigoplus_{\{\alpha\in Q_1\; | \; t(\alpha)=i\}} kQe_{s(\alpha)} \longrightarrow T_i^L\longrightarrow 0,
  \end{align}
where $\eta^Q_i(p)=\sum_{\{\alpha\in Q_1\; | \; t(\alpha)=i\}} p\alpha$ for any path $p$ starting at $i$.

The following result relates Leavitt path algebras to graded universal localizations; compare \cite{AB, Sm}.

\begin{prop}\label{prop:Leavittloc}
Keep the notation as above. Then we have the following statements:
\begin{enumerate}
\item the graded algebra homomorphism $\iota_Q\colon kQ\rightarrow L(Q)$ is a graded universal localization with respect to $\{\xi^Q_i\; |\; i\in Q_0 \mbox{ that are not sinks}\}$;
\item  the graded algebra homomorphism $\kappa_Q\colon kQ \rightarrow L(Q^{\rm op})^{\rm op}$ is a graded universal localization with respect to $\{\eta^Q_i\; |\; i\in Q_0 \mbox{ that are not sources}\}$.
\end{enumerate}
\end{prop}

\begin{proof}
The statement (1) is contained in  \cite[5.3]{Sm}, where the homomorphism $\iota_Q$ is proved to be a universal localization of algebras. Then it follows immediately that it is a graded universal localization.

For (2), we recall a general fact; compare \cite[p.52]{Sch}. Let  $\theta\colon R\rightarrow S$ be a graded universal localization of graded algebras with respect to  a family $\Sigma$ of morphisms. Then $\theta^{\rm op}\colon R^{\rm op}\rightarrow S^{\rm op}$ is a graded universal localization with respect to $\Sigma^{\mathrm{tr}}=\{\xi^{\mathrm{tr}}\; |\: \xi\in \Sigma\}$. Here, $(-)^{\mathrm{tr}}=\bigoplus_{n\in \mathbb{Z}}{\rm Hom}_{R\mbox{-}{\rm Gr}}(-, R(n))\colon R\mbox{-grproj}\rightarrow R^{\rm op}\mbox{-grproj}$ is the duality functor on finitely generated graded projective modules. Notice that $(kQ^{\rm op}e_i(n))^{\mathrm{tr}}\simeq kQe_i(-n)$ for any vertex $i$ and $n\in \mathbb{Z}$. Then for a vertex $i$ in $Q$ that is not a source, we may identify  $\eta^Q_i(1)$ with $(\xi^{Q^{\rm op}}_i)^{\mathrm{tr}}$. Then (2) follows from (1).
\end{proof}

\section{The Gorenstein projective modules over trivial extensions}

In this section, we study the stable category of Gorenstein projective modules over a trivial extension of an algebra by an invertible bimodule, which is related to the derived category of a strongly graded algebra.

\subsection{Gorenstein projective modules} We first recall some results on Gorenstein projective modules.

 Let $\Lambda$ be an algebra. Recall from \cite[p.400]{AM} that a
 complex $X$ of projective $\Lambda$-modules is \emph{totally acyclic} if it is acyclic
 and for any projective $\Lambda$-module $Q$ the complex ${\rm
Hom}_\Lambda(X, Q)$ is acyclic. A $\Lambda$-module $G$ is called \emph{Gorenstein projective}
if there is a totally acyclic complex $X$ such that the
first cocycle $Z^1(X)$ is isomorphic to $G$ (\cite{EJ95}), in which case
the complex $X$ is said to be a \emph{complete resolution}
of $G$. Indeed, any cocycle $Z^i(X)$ is Gorenstein projective. We denote by $\Lambda\mbox{-GProj}$ the full subcategory  of
$\Lambda\mbox{-Mod}$ consisting of Gorenstein projective modules. Observe
that projective modules are Gorenstein projective.

Recall that the full subcategory $\Lambda\mbox{-GProj}$ of $\Lambda\mbox{-Mod}$ is closed under extensions. It follows
that  $\Lambda\mbox{-GProj}$ is an exact category in the sense of Quillen. Moreover, it is a Frobenius category, whose projective-injective objects are precisely projective $\Lambda$-modules. Then by \cite[Theorem I.2.8]{Hap88} the stable category   $\Lambda\mbox{-\underline{GProj}}$ modulo projective $\Lambda$-modules has a natural triangulated structure: the translation functor is a quasi-inverse of the syzygy functor, and triangles are induced by short exact sequences with terms in  $\Lambda\mbox{-GProj}$.

The stable category $\Lambda\mbox{-\underline{GProj}}$ is related to the homotopy category of complexes. Denote by $\Lambda\mbox{-Proj}$ the subcategory of $\Lambda\mbox{-Mod}$ consisting of projective modules, and  by $\mathbf{K}(\Lambda\mbox{-Proj})$ the homototpy category of complexes of projective modules. Let $\mathbf{K}_{\rm tac}(\Lambda\mbox{-Proj})$ be the triangulated subcategory of $\mathbf{K}(\Lambda\mbox{-Proj})$ consisting of totally acyclic complexes.

The following result is the dual of \cite[Proposition 7.2]{Kra}.

\begin{lem}\label{lem:GProj}
Keep the notation as above. Then there is a triangle equivalence
$$Z^1\colon \mathbf{K}_{\rm tac}(\Lambda\mbox{-}{\rm Proj}) \stackrel{\sim}\longrightarrow \Lambda\mbox{-}\underline{\rm GProj},$$
sending a complex $X$ to its first cocycle $Z^1(X)$. Its quasi-inverse sends a Gorenstein projective module to its complete resolution. \hfill $\square$
\end{lem}

 We recall the full subcategory $\Lambda\mbox{-Gproj}$ of $\Lambda\mbox{-Mod}$ as follows.   A $\Lambda$-module $G$ lies in $\Lambda\mbox{-Gproj}$ if and only if there is an acyclic complex $P$ of finitely generated projective $\Lambda$-module such that the dual $P^{\mathrm{tr}}={\rm Hom}_\Lambda(P, \Lambda)$ is acyclic and $Z^1(P)\simeq G$. In this case, the complex $P$ is a complete resolution of $G$, and thus $\Lambda\mbox{-Gproj}\subseteq \Lambda\mbox{-GProj}$. We mention that for a left coherent algebra $\Lambda$,  there is an equality  $\Lambda\mbox{-Gproj}=\Lambda\mbox{-GProj}\cap \Lambda\mbox{-mod}$; see \cite[Lemma 3.4]{Ch11'}. Here, $\Lambda\mbox{-mod}$ denotes the category of finitely presented $\Lambda$-modules.

   The subcategory $\Lambda\mbox{-Gproj}$ of $\Lambda\mbox{-Mod}$ is closed under extensions, and thus becomes an exact category. Moreover, it is a Frobenius category whose projective-injective objects are precisely finitely generated projective $\Lambda$-modules. The corresponding stable category  $\Lambda\mbox{-\underline{Gproj}}$ has a natural triangulated structure.

We observe that the stable category $\Lambda\mbox{-\underline{GProj}}$ has arbitrary coproducts, and that the  inclusion $\Lambda\mbox{-Gproj}\subseteq \Lambda\mbox{-GProj}$ induces an embedding
$$\Lambda\mbox{-\underline{Gproj}}\subseteq (\Lambda\mbox{-\underline{GProj}})^c$$
of triangulated categories. Here, we recall that $(\Lambda\mbox{-\underline{GProj}})^c$ denotes the thick subcategory of $\Lambda\mbox{-\underline{GProj}}$ consisting of compact objects. The embedding $\Lambda\mbox{-\underline{Gproj}}\subseteq (\Lambda\mbox{-\underline{GProj}})^c$ induces further an embedding $$\widetilde{\Lambda\mbox{-\underline{Gproj}}}\subseteq (\Lambda\mbox{-\underline{GProj}})^c,$$
 where $\widetilde{\Lambda\mbox{-\underline{Gproj}}}$ is the idempotent completion \cite{BS} of $\Lambda\mbox{-\underline{Gproj}}$.

\subsection{Trivial extension} We will study modules over a trivial extension $\Lambda=B\ltimes X$ of an algebra $B$ by an invertible $B$-bimodule $X$. Observe that there exists a strongly graded algebra $A=\bigoplus_{n\in \mathbb{Z}} A^n$ such that $A^0=B$ and $A^{-1}=X$ as $A$-bimodules. In what follows, we begin with a strongly graded algebra $A$ and consider the corresponding trivial extension $\Lambda=A^0\ltimes A^{-1}$.

Let $A=\bigoplus_{n\in \mathbb{Z}} A^n$ be a strongly graded algebra. Then each $A^0$-bimodule $A^n$ is invertible and the multiplication induces an isomorphism
\begin{align}\label{equ:tensor}
A^n\otimes_{A^0} A^m\stackrel{\sim}\longrightarrow A^{n+m}
\end{align}
of $A^0$-bimodules. This isomorphism sends $a^n\otimes a^m$ to $a^na^m$ for $a^n\in A^n$ and $a^m\in A^m$.

We observe the following result.

\begin{lem}
Let $M$ be an $A^0$-module, and let $n, m\in \mathbb{Z}$. Then the following hold:
\begin{enumerate}
\item there is an isomorphism
\begin{align} \label{equ:iso1}
\theta^n_M\colon A^n\otimes_{A^0} M \stackrel{\sim}\longrightarrow {\rm Hom}_{A^0}(A^{-n}, M)
\end{align}
such that $\theta^n_M(a^n\otimes m) (a^{-n})=(a^{-n}a^n).m$;
\item there is an isomorphism
\begin{align}\label{equ:iso}
\theta^{n, m}\colon A^{m-n}\stackrel{\sim}\longrightarrow {\rm Hom}_{A^0}(A^n, A^m)
\end{align}
such that $\theta^{n, m}(a)(b)=ba$ for $a\in A^{m-n}$ and $b\in A^n$.
\end{enumerate}
\end{lem}

\begin{proof}
For (1), we use the fact that the endofunctor $A^n\otimes_{A^0}-$ on $A^0\mbox{-Mod}$ is isomorphic to  ${\rm Hom}_{A^0}(A^{-n}, -)$; both are quasi-inverse to $A^{-n}\otimes_{A^0}-$. The isomorphism in (2) follows from (1), once we recall from (\ref{equ:tensor}) the isomorphism $A^{m-n}\simeq A^{-n}\otimes_{A^0} A^m$.
\end{proof}

Let $\Lambda=A^0\ltimes A^{-1}$ be the \emph{trivial extension} of $A^0$ by the $A^0$-bimodule $A^{-1}$. That is, an element in $\Lambda$ is given by $(a, b)$ with $a\in A^0$ and $b\in A^{-1}$, and the multiplication is defined such that $(a, b)\cdot (a', b')=(aa', ab'+ba')$.  We have the natural embedding $A^0\rightarrow \Lambda$ by identifying $a$ with $(a, 0)$. We observe that $\Lambda$ is a finitely generated projective $A^0$-module on each side.

  We observe that $A^{-1}$ is a two-sided ideal of $\Lambda$ and that $\Lambda/A^{-1}\simeq A^0$. In this manner, we view $A^0$ as a $\Lambda$-module. We claim that $A^0$ lies in $\Lambda\mbox{-Gproj}$.

   For the claim,  we will construct a complex $P$ of finitely generated projective $\Lambda$-modules. For each integer $n$, we define $P^{n}=A^{n}\oplus A^{n-1}$, which carries a natural $\Lambda$-module structure induced by the multiplication of $A$. More precisely, for $(a, b)\in \Lambda$ and $(a^{n}, a^{n-1})\in P^n$, the action is given such that $(a, b).(a^{n}, a^{n-1})=(aa^n, aa^{n-1}+ba^{n})$. We observe that there is an isomorphism $P^n\simeq \Lambda\otimes_{A^0} A^{n}$ of $\Lambda$-modules. Recall that the $A^0$-bimodule $A^{n}$ is invertible. Then each $\Lambda$-module $P^n$ is  finitely generated projective. We define the differential $d^n\colon P^n\rightarrow P^{n+1}$ such that $d^n((a^{n}, a^{n-1}))=(0, a^{n})$. Then the complex $P$ is acyclic with $Z^{1}(P)\simeq A^0$.

   Consider the following  isomorphisms of right $\Lambda$-modules
  $${\rm Hom}_\Lambda(P^n, \Lambda)\simeq {\rm Hom}_\Lambda(\Lambda\otimes_{A^0} A^{n}, \Lambda)\simeq {\rm Hom}_{A^0}(A^{n}, \Lambda)\simeq A^{-n}\oplus A^{-n-1}.$$
Here, the rightmost isomorphism is obtained by applying (\ref{equ:iso}) twice, and the right $\Lambda$-module structure on $A^{-n}\oplus A^{-n-1}$ is induced from the multiplication of $A$. Via these composite isomorphisms, the differential of the dual complex $(P)^{\mathrm{tr}}={\rm Hom}_\Lambda (P, \Lambda)$ sends  $(a^{-n-1}, a^{-n-2})$ in $(P^{n+1})^{\mathrm{tr}}$ to $(-1)^n(0, a^{-n-1})$ in $(P^n)^{\mathrm{tr}}$. In particular, $P^{\mathrm{tr}}$ is acyclic. This proves that $A^0$ lies in $\Lambda\mbox{-Gproj}$. Moreover, the complex $P$ is a complete resolution of $A^0$. The differential $d^0\colon P^0\rightarrow P^1$ induces a monomorphism $d\colon A^0\rightarrow P^1$ such that $d(a)=(0, a)$. We observe that $d$ is a left $\Lambda\mbox{-Proj}$-approximation, which means that any $\Lambda$-module homomorphism $A^0\rightarrow Q$ with $Q$ projective factors through $d$.

Let $M$ be a $\Lambda$-module. We will compute the Hom space $\underline{\rm Hom}_\Lambda(A^0, M)$ in the stable module category $\Lambda\mbox{-\underline{Mod}}$. Recall that $\underline{\rm Hom}_\Lambda(A^0, M)={\rm Hom}_\Lambda(A^0, M)/\mathcal{P}$, where $\mathcal{P}$ denotes the subspace consisting of homomorphisms from $A^0$ to $M$ that factor though projectives.

The $A^{-1}$-action on $M$ induces a homomorphism of  $A^0$-modules $$\phi_M\colon A^{-1}\otimes_{A^0} M\rightarrow M,$$
 which satisfies $\phi_M\circ ({\rm Id}_{A^{-1}}\otimes \phi_M)=0$; more explicitly, $\phi_M(a^{-1}\otimes m)=(0, a^{-1}).m$.  There is a unique $A^0$-submodule $K_M$ of $M$ with the property $A^{-1}\otimes_{A^0} K_M={\rm Ker}\; \phi_M$. Here, we use the following general fact: since the $A^0$-bimodule $A^{-1}$ is invertible, each submodule of  $A^{-1}\otimes_{A^0} M$ is of the form $A^{-1}\otimes_{A^0} M'$ for an $A^0$-submodule $M'$ of $M$. Hence, we have the following exact sequence of $A^0$-modules:
\begin{align}\label{equ:K_M}
0\longrightarrow A^{-1}\otimes_{A^0} K_M\longrightarrow A^{-1}\otimes_{A^0} M\stackrel{\phi_M}\longrightarrow M.
\end{align}
The condition $\phi_M\circ ({\rm Id}_{A^{-1}}\otimes \phi_M)=0$ implies that ${\rm Im}\; \phi_M\subseteq K_M$. Indeed, $K_M=\{m\in M\; |\; A^{-1}.m=0\}$.

\begin{lem}\label{lem:stablehom}
Keep the notation as above. Then there is an isomorphism
$$\underline{\rm Hom}_\Lambda(A^0, M)\simeq K_M/{{\rm Im}\; \phi_M}.$$
\end{lem}

We apply the isomorphism to the case $M=A^0$. Then we obtain an isomorphism of algebras
\begin{align}\label{equ:isoA}
(A^0)^{\rm op} \stackrel{\sim}\longrightarrow \underline{\rm End}_\Lambda(A^0).
\end{align}
Here, $\phi_M=0$ and $K_M=A^0$. This isomorphism sends $a^0\in A^0$ to the endomorphism $r_{a^0}\colon A^0\rightarrow A^0$ given by  $r_{a^0}(x)=xa^0$.

\begin{proof}
Recall that $A^0=\Lambda/{A^{-1}}$. Then there is an isomorphism
$$\phi\colon {\rm Hom}_\Lambda(A^0, M)\stackrel{\sim}\longrightarrow K_M$$
sending $f$ to $f(1)$. It suffices to show that $\phi$ identifies $\mathcal{P}$ with ${\rm Im}\; \phi_M$. Recall that $d\colon A^0\rightarrow P^1$, that sends $a^0$ to $(0, a^0)\in P^1$,  is a left $\Lambda\mbox{-Proj}$-approximation. It follows that $\mathcal{P}$ is the image of  ${\rm Hom}_\Lambda(d, M)\colon {\rm Hom}_\Lambda(P^1, M)\rightarrow {\rm Hom}_\Lambda(A^0, M)$.

Recall the isomorphism $\psi\colon\Lambda\otimes_{A^0} A^1\simeq P^1$, sending $(a, b)\otimes a^0$ to $(aa^0, ba^0)$. Assume that the isomorphism $A^0\simeq A^{-1}\otimes_{A^0} A^1$ of $A^0$-bimodules, which is induced by the multiplication of $A$,  sends $1$ to $\sum_{i=1}^r a_i\otimes b_i$; moreover, we have $1=\sum_{i=1}^r a_ib_i$; compare (\ref{equ:tensor}). Then we have $$\psi^{-1}((a^1, a^0))= (1, 0)\otimes a^1+ \sum_{i=1}^r(0, a^0a_i)\otimes b_i.$$

Consider the following isomorphisms
$$\Phi\colon A^{-1}\otimes_{A^0} M\stackrel{\theta_M^{-1}}\rightarrow {\rm Hom}_{A^0}(A^1, M)\simeq {\rm Hom}_\Lambda(\Lambda\otimes_{A^0} A^1, M)\simeq {\rm Hom}_\Lambda(P^1, M)$$
where the isomorphism $\theta_M^{-1}$ is given in (\ref{equ:iso1}) and the rightmost isomorphism is given by ${\rm Hom}_\Lambda(\psi^{-1}, M)$. By direct calculation,  we infer that the composite isomorphism $\Phi$ satisfies $$\Phi(a^{-1}\otimes m)((a^1, a^0))=(a^1a^{-1}, a^0a^{-1}).m.$$
 Then ${\rm Hom}_\Lambda(d, M)\circ \Phi\colon A^{-1}\otimes_{A^0} M\rightarrow {\rm Hom}_\Lambda (A^0, M)$ sends $a^{-1}\otimes m$ to a map, that sends $a^0\in A^0$ to $(0, a^0a^{-1}).m$. Hence, the composite $\phi\circ {\rm Hom}_\Lambda(d, M)\circ \Phi$ coincides with $\phi_M$. In particular, $\phi$ identifies the image of ${\rm Hom}_\Lambda(d, M)$, which  is just shown to be $\mathcal{P}$, with ${\rm Im}\; \phi_M$. Then we are done. \end{proof}

\subsection{The stable category} Let $\Lambda=A^0\ltimes A^{-1}$ be as above. Recall that the stable category $\Lambda\mbox{-\underline{GProj}}$ of Gorenstein projective $\Lambda$-modules is a triangulated category with arbitrary coproducts.
Recall that $A^0\in \Lambda\mbox{-\underline{Gproj}}\subseteq \Lambda\mbox{-\underline{GProj}}$. We denote by  ${\rm thick} \langle A^0\rangle $  the smallest thick subcategory that contains $A^0$, and by ${\rm Loc}\langle A^0 \rangle$ the smallest triangulated subcategory that contains $A^0$ and is closed under arbitrary coproducts.

Recall the embedding $\widetilde{\Lambda\mbox{-\underline{Gproj}}}\subseteq (\Lambda\mbox{-\underline{GProj}})^c$, and observe that ${\rm thick} \langle A^0\rangle\subseteq \widetilde{\Lambda\mbox{-\underline{Gproj}}}$. For an object $X$ in an additive category, we denote by ${\rm add}\; X$ the full subcategory consisting of direct summands of direct sums of finitely many copies of $X$.

\begin{prop}\label{prop:hered}
Let $A=\bigoplus_{n\in \mathbb{Z}} A^n$ be a strongly graded algebra, and let $\Lambda=A^0\ltimes A^{-1}$ be the trivial extension. Assume that the algebra $A^0$ is left hereditary. Then the following hold:
\begin{enumerate}
\item the stable category $\Lambda\mbox{-\underline{\rm GProj}}$ is compactly generated with $A^0$ a compact generator;
\item $\Lambda\mbox{-\underline{\rm GProj}}={\rm Loc}\langle A^0 \rangle$ and $(\Lambda\mbox{-\underline{\rm GProj}})^c={\rm thick}\langle A^0 \rangle$, which is equivalent to $\widetilde{\Lambda\mbox{-\underline{\rm Gproj}}}$, the idempotent completion of $\Lambda\mbox{-\underline{\rm Gproj}}$;
\item if in addition $A^0$ is von Neumann regular, then $\Lambda\mbox{-}{\rm Gproj}={\rm add}\; (A^0\oplus \Lambda)$, and there are  equivalences of triangulated categories
$$(\Lambda\mbox{-\underline{\rm GProj}})^c \stackrel{\sim}\longrightarrow \Lambda\mbox{-\underline{\rm Gproj}}\stackrel{\sim}
\longrightarrow (A^0\mbox{-}{\rm proj}, A^1\otimes_{A^0}-).$$
\end{enumerate}
\end{prop}

We recall that if $A^0$ is von Neumann regular,  the category $A^0\mbox{-}{\rm proj}$ of finitely generated projective $A^0$-modules is semisimple abelian and that
$A^{1}\otimes_{A^0}-$ is an auto-equivalence on it. Thus we have the triangulated category $(A^0\mbox{-}{\rm proj}, A^{1}\otimes_{A^0}-)$; see Section 4.

\begin{proof}
Recall that $A^0$ is compact in $\Lambda\mbox{-\underline{GProj}}$. We claim that for a Gorenstein projective $\Lambda$-module $M$, $\underline{\rm Hom}_\Lambda (A^0, M)=0$ implies that $M$ is projective. Then (1) follows, and (2) is an immediate consequence of (1).

For the claim, we observe that any projective $\Lambda$-module is projective as an $A^0$-module. Since $M$ is a submodule of a projective $\Lambda$-module and $A^0$ is left hereditary, the underlying $A^0$-module $M$ and then ${\rm Im}\; \phi_M$ is projective. Then the following short exact sequence induced by  (\ref{equ:K_M})
$$
0\longrightarrow A^{-1}\otimes_{A^0} K_M\longrightarrow A^{-1}\otimes_{A^0} M\stackrel{\phi_M}\longrightarrow {\rm Im}\; \phi_M\longrightarrow 0$$
splits. In particular, there is a decomposition $M=K'\oplus K_M$ of $A^0$-modules such that $\phi_M$ induces an isomorphism $A^{-1}\otimes_{A^0} K'\simeq {\rm Im}\; \phi_M$. By Lemma \ref{lem:stablehom}, $\underline{\rm Hom}_\Lambda(A^0, M)=0$  implies that ${\rm Im}\;\phi_M=K_M$. This implies that $$K'\oplus (A^{-1}\otimes_{A^0} K')\stackrel{\sim}\longrightarrow M,$$
which implies that $M$ is isomorphic to $\Lambda\otimes_{A^0}K'$. Note that the $A^0$-module $K'$ is projective. It follows that the $\Lambda$-module $M$ is projective.

For (3), let $M\in \Lambda\mbox{-Gproj}$. We will show that $M$ lies in ${\rm add}\; (A^0\oplus \Lambda)$. Observe that $M$, $K_M$ and ${\rm Im}\; \phi_M$, as $A^0$-modules, are finitely generated projective. Since $A^0$ is von Neumann regular, ${\rm Im}\; \phi_M\subseteq K_M$ is a direct summand. Hence, we have a decomposition $K_M={\rm Im}\; \phi_M\oplus K''$ of $A^0$-modules. Then we have  $$M=(K'\oplus {\rm Im}\; \phi_M)\oplus K''.$$
 By the isomorphism $A^{-1}\otimes_{A^0} K'\simeq {\rm Im}\; \phi_M$, the $A^0$-submodule $K'\oplus {\rm Im}\; \phi_M$ of $M$ is indeed a $\Lambda$-submodule, which is isomorphic to $\Lambda\otimes_{A^0} K'$ and then belongs to ${\rm add}\; \Lambda$. The $A^{-1}$-action  on $K_M$ and then on $K''$ is trivial. In particular, $K''\subseteq M$ is a $\Lambda$-submodule. Since $K''$, as an $A^0$-module, is finitely generated projective, we have that $K''$ belongs to ${\rm add}\; A^0$.

For the triangle equivalences in (3), we consider the natural functor $$F\colon A^0\mbox{-proj}\longrightarrow \Lambda\mbox{-\underline{\rm Gproj}},$$ where we view a finitely generated projective $A^0$-module $Q$ as a $\Lambda$-module, on which $A^{-1}$ acts trivially. This functor is well defined, since $A^0$ and thus each $Q$ lie in $\Lambda\mbox{-{\rm Gproj}}$. The functor $F$ is fully faithful by the isomorphism (\ref{equ:isoA}). By the above, it is dense, and thus an equivalence of categories.

It remains to show that via the equivalence $F$, the translation functor on $\Lambda\mbox{-}\underline{\mbox{Gproj}}$ corresponds to the functor $A^1\otimes_{A^0}-$ on $A^0\mbox{-proj}$. For this, recall that the syzygy functor is a quasi-inverse of the translation functor on the stable category $\Lambda\mbox{-\underline{\rm Gproj}}$. We observe an exact sequence $$0\rightarrow A^{-1}\otimes_{A^0} Q\rightarrow \Lambda\otimes_{A^0} Q\rightarrow Q\rightarrow 0$$ for each $A^0$-module $Q$. This implies that the syzygy functor on $\Lambda\mbox{-\underline{\rm Gproj}}$ corresponds to the functor  $A^{-1}\otimes_{A^0}-$ on $A^0\mbox{-proj}$, and thus the translation functor corresponds to $A^{1}\otimes_{A^0}-$.

We infer from the equivalence $F$ that the category $\Lambda\mbox{-}\underline{\mbox{Gproj}}$ has split idempotents, or equivalently,  it is equivalent to the idempotent completion $\widetilde{\Lambda\mbox{-\underline{\rm Gproj}}}$. Then the remaining equivalence follows from (2).
\end{proof}

The following result relates the stable category of Gorenstein projective modules to the derived category of the strongly graded algebra.

\begin{prop}\label{prop:loc}
Let $A=\bigoplus_{n\in \mathbb{Z}} A^n$ be a strongly graded algebra, and let $\Lambda=A^0\ltimes A^{-1}$ be the trivial extension.  Then there is a triangle equivalence
$${\rm Loc}\langle A^0 \rangle\stackrel{\sim}\longrightarrow \mathbf{D}(A),$$
sending $A^0$ to $A$. If $A^0$ is left hereditary, there is a triangle equivalence
$$\Lambda\mbox{-\underline{\rm GProj}} \stackrel{\sim}\longrightarrow \mathbf{D}(A).$$
\end{prop}

If $A^0$ is von Neumann regular and left hereditary, the restriction of the above equivalence to compact objects yields an equivalence
$$\Lambda\mbox{-\underline{\rm Gproj}} \stackrel{\sim}\longrightarrow {\rm perf}(A),$$
which is further equivalent to $(A^0\mbox{-}{\rm proj}, A^{1}\otimes_{A^0}-)$; compare Proposition \ref{prop:hered}(3) and Lemma \ref{lem:perfder}.

\begin{proof}
By Lemma \ref{lem:GProj}, we identify $\Lambda\mbox{-\underline{GProj}}$ with $\mathbf{K}_{\rm tac}(\Lambda\mbox{-}{\rm Proj})$, and in this way $A^0$ is identified with the complex $P$. Then we have a triangle equivalence ${\rm Loc}\langle A^0 \rangle\stackrel{\sim}\longrightarrow {\rm Loc}\langle P \rangle$.

The complex $P$ is naturally a dg $\Lambda$-$A$-bimodule. Indeed, $A$ acts on $P$ as follows: for $a'^m\in A^m$ and $(a^{n}, a^{n-1})\in P^n$, we have $(a^{n}, a^{n-1}).a'^m=(a^{n}a'^m, a^{n-1}a'^m) \in P^{n+m}$. We claim that this dg bimodule is right quasi-balanced and then by Proposition \ref{prop:triangle-equi} we have a triangle equivalence ${\rm Loc}\langle P \rangle\stackrel{\sim}\longrightarrow \mathbf{D}(A)$. Then the first statement follows. The second follows from the first and Proposition \ref{prop:hered}(2).

For the claim, we compute the homomorphism $\Phi\colon A\rightarrow {\rm End}_\Lambda(P)^{\rm opp}$ of dg algebras induced by the right $A$-action on $P$. More precisely, $\Phi(a)=r_a$ for homogeneous elements $a$ in $A$; see Subsection 2.2.   The homogeneous component of degree $n$ of ${\rm End}_\Lambda(P)$ is given as follows. \begin{align*}
{\rm End}_\Lambda(P)^n & =\prod_{p\in \mathbb{Z}} {\rm Hom}_\Lambda(P^p, P^{p+n})\simeq \prod_{p\in \mathbb{Z}} {\rm Hom}_\Lambda(\Lambda\otimes_{A^0}A^{p}, P^{p+n})\\
& \simeq \prod_{p\in \mathbb{Z}} {\rm Hom}_{A^0}(A^{p}, P^{p+n})\simeq \prod_{p\in \mathbb{Z}} (A^{n}\oplus A^{n-1}).
\end{align*}
The last isomorphism is induced by (\ref{equ:iso}); here, we recall that $P^{p+n}=A^{p+n}\oplus A^{p+n-1}$. Then via the above isomorphism, the elements in ${\rm End}_\Lambda(P)^n$ are denoted by $\{(x_p^n, y_p^n)\}_{p\in \mathbb{Z}}$, where $x_p^n\in A^{n}$ and $y_p^n\in A^{n-1}$. The differential  $d^n\colon {\rm End}_\Lambda(P)^n\rightarrow {\rm End}_\Lambda(P)^{n+1}$ sends $\{(x_p^n, y_p^n)\}_{p\in \mathbb{Z}}$ to $\{(0, x_p^n-(-1)^n x_{p+1}^n)\}_{p\in \mathbb{Z}}$. The homomorphism $\Phi$ sends an element $a$ in $A^n$ to $\{((-1)^{pn}a, 0)\}_{p\in \mathbb{Z}}$ in ${\rm End}_\Lambda(P)^n$. Then direct calculation shows that $\Phi$ is a quasi-isomorphism. \end{proof}

\section{Proof of Main Theorem}

In this section, we combine the results in the previous sections to give a proof of Main Theorem in the introduction.

Let $Q$ be a finite quiver. Denote by $A=kQ/J^2$ the corresponding finite-dimensional algebra with radical square zero.
Recall the embedding ${\bf i}\colon \mathbf{D}^b(A\mbox{-mod})\rightarrow \mathbf{K}(A\mbox{-Inj})$ sending a bounded complex $X$ to its homotopically injective resolution ${\bf i} X$. We apply \cite[Lemma 2.1]{Kra} and the isomorphism (\ref{equ:2}) to deduce the following isomorphism for any complex $X$ in $\mathbf{K}(A\mbox{-Inj})$
\begin{align}\label{equ:isoK}
{\rm Hom}_{\mathbf{K}(A\mbox{-}{\rm Inj})}({\bf i} A, X)\simeq {\rm Hom}_{\mathbf{K}(A\mbox{-}{\rm Mod})}(A, X)\simeq H^0(X).
\end{align}
In particular, we infer that $\mathbf{K}_{\rm ac}(A\mbox{-Inj})=({\bf i} A)^\perp$, where $({\bf i} A)^\perp$ denotes the right perpendicular subcategory of $\mathbf{K}(A\mbox{-Inj})$ with respect to ${\bf i}A$, that is,
\[({\bf i}A)^\perp = \{X\in \mathbf{K}(A\mbox{-Inj})\mid \mathrm{Hom}_{\mathbf{K}(A\mbox{-}\mathrm{Inj})}({\bf i}A, X[n])=0\text{ for any }n\in\mathbb{Z}\}.\]

Following \cite{Buc, Or04}, the \emph{singularity category} $\mathbf{D}_{\rm sg}(A)$ of $A$ is the Verdier quotient triangulated category of $\mathbf{D}^b(A\mbox{-mod})$ with respect to ${\rm perf}(A)$; also see \cite{Hap91}. Here, we recall that ${\rm perf}(A)$ is triangle equivalent to $\mathbf{K}^b(A\mbox{-proj})$, the homotopy category of bounded complexes of finitely generated projective $A$-modules. Recall from \cite[Corollary 5.4]{Kra} a triangle  equivalence
\begin{align}\label{equ:tri}
\mathbf{D}_{\rm sg}(A) \stackrel{\sim}\longrightarrow \mathbf{K}_{\rm ac}(A\mbox{-}{\rm Inj})^c.
 \end{align}
 Here, we implicitly use  the fact that $\mathbf{D}_{\rm sg}(A)$ has split idempotents; see \cite[Corollary 2.4]{Ch11}.

The following result contains the first part of Main Theorem. Recall that the Leavitt path algebra $L(Q)=\bigoplus_{n\in \mathbb{Z}} L(Q)^n$ is graded, and that $\Lambda^{+}(Q)=L(Q)^0\ltimes L(Q)^{1}$  denotes the trivial extension of $L(Q)^0$ by the $L(Q)^0$-bimodule $L(Q)^{1}$.

\begin{thm}\label{thm:A}
Let $Q$ be a finite quiver without sinks. Then there are triangle equivalences
$$\mathbf{K}_{\rm ac}(kQ/J^2\mbox{-}{\rm Inj})\stackrel{\sim}\longrightarrow\mathbf{D}(L(Q)^{\rm op}) \stackrel{\sim}\longrightarrow \Lambda^+(Q)\mbox{-}\underline{\rm GProj}.$$
Restricting these equivalences to compact objects, we obtain triangle equivalences
$$\mathbf{D}_{\rm sg}(kQ/J^2)\stackrel{\sim}\longrightarrow {\rm perf}(L(Q)^{\rm op}) \stackrel{\sim}\longrightarrow \Lambda^+(Q)\mbox{-}\underline{\rm Gproj},$$
which is further equivalent to  $(L(Q)^{\rm op}\mbox{-}{\rm grproj}, (1))\stackrel{\sim}\longrightarrow (L(Q)\mbox{-}{\rm grproj}, (-1))$.
\end{thm}

The equivalence $\mathbf{D}_{\rm sg}(kQ/J^2)\stackrel{\sim}\longrightarrow (L(Q)\mbox{-}{\rm grproj}, (-1))$ is essentially contained in \cite[Theorem 7.2]{Sm}; compare \cite[Theorem B]{Ch11}. Here, we identify ${\rm qgr} (kQ)$ with $L(Q)\mbox{-}{\rm grproj}$ by \cite[Theorem 5.9]{Sm}.

\begin{proof}
Set $A=kQ/J^2$. Recall from Theorem \ref{prop:Koszul} that there is a triangle equivalence $F\colon \mathbf{K}(A\mbox{-}{\rm Inj})\stackrel{\sim}\longrightarrow\mathbf{D}(kQ^{\rm op})$, that sends ${\bf i} A$ to $\bigoplus_{i\in Q_0} T_i$. Here, we recall that $A=\bigoplus_{i\in Q_0} P_i$ with $P_i$ the indecomposable projective $A$-module corresponding to $i$, and the graded $kQ^{\rm op}$-module $T_i$ is defined in (\ref{equ:4}). Hence, $F$ identifies $\mathbf{K}_{\rm ac}(A\mbox{-Inj})=({\bf i} A)^\perp$ with the right perpendicular subcategory $(\bigoplus_{i\in Q_0} T_i)^\perp=\{T_i\; |\; i \in Q_0\}^\perp$ in $\mathbf{D}(kQ^{\rm op})$.

Recall from Proposition \ref{prop:Leavittloc}(2) that $\kappa_{Q^{\rm op}}\colon kQ^{\rm op}\rightarrow L(Q)^{\rm op}$ is a graded universal localization with respect to $\{\eta^{Q^{\rm op}}_i\; |\; i\in Q_0\}$; here, we use the assumption that $Q^{\rm op}$ has no sources. Recall the exact sequence (\ref{equ:4}) that defines $T_i$. Then Proposition \ref{prop:perpD} yields a triangle equivalence $\{T_i\; |\; i \in Q_0\}^\perp\stackrel{\sim}\longrightarrow \mathbf{D}(L(Q)^{\rm op})$. Then we obtain the desired triangle equivalence $\mathbf{K}_{\rm ac}(A\mbox{-}{\rm Inj})\stackrel{\sim}\longrightarrow\mathbf{D}(L(Q)^{\rm op})$.

We apply Lemma \ref{lem:Leavitt} and then Proposition \ref{prop:loc} to obtain a triangle equivalence $\mathbf{D}(L(Q)^{\rm op})\simeq \Lambda'\mbox{-}\underline{\rm GProj}$, where $\Lambda'=(L(Q)^{\rm op})^0\ltimes (L(Q)^{\rm op})^{-1}$ denotes the trivial extension. Recall the involution $(-)^*\colon L(Q)\rightarrow L(Q)$, which induces an isomorphism $\Lambda'\simeq \Lambda^+(Q)$. Hence, we have the triangle equivalence $\mathbf{D}(L(Q)^{\rm op})\simeq \Lambda^+(Q)\mbox{-}\underline{\rm GProj}$.

In view of the equivalence (\ref{equ:tri}), Proposition \ref{prop:hered} and Lemma \ref{lem:Leavitt}(3) and (4), the equivalences on the subcategories of compact objects follow immediately.
\end{proof}

We now consider $\mathbf{K}(A\mbox{-Proj})$ the homotopy category of complexes of projective $A$-modules. By the isomorphism (\ref{equ:2}), we have $\mathbf{K}_{\rm ac}(A\mbox{-Proj})=A^\perp$, the right perpendicular subcategory of $\mathbf{K}(A\mbox{-Proj})$ with respect to $A$. Applying \cite[Propositions 7.14 and 7.12]{Nee08} and the localization theorem in \cite[8.1.5]{Kel98}, we have that the category $\mathbf{K}_{\rm ac}(A\mbox{-Proj})$ is compactly generated with a triangle equivalence
\begin{align}\label{equ:triP}
 \mathbf{D}_{\rm sg}(A^{\rm op})^{\rm op} \stackrel{\sim}\longrightarrow \mathbf{K}_{\rm ac}(A\mbox{-Proj})^c.
 \end{align}
Here, for any category $\mathcal{C}$, we denote by $\mathcal{C}^{\rm op}$ its opposite category.

Recall the Nakayama functors $\nu=DA\otimes_A-\colon A\mbox{-Proj}\stackrel{\sim}\longrightarrow A\mbox{-Inj}$ is an equivalence, which sends $A$ to $DA$. Thus we have a triangle equivalence $\mathbf{K}(A\mbox{-Proj})\stackrel{\sim}\longrightarrow \mathbf{K}(A\mbox{-Inj})$ sending $A$ to $DA$;  compare \cite[Example 2.6]{Kra}. We conclude with a triangle equivalence
\begin{align}\label{equ:perp}
\mathbf{K}_{\rm ac}(A\mbox{-Proj}) \stackrel{\sim}\longrightarrow (DA)^\perp.
 \end{align}
Here, $(DA)^\perp$ is the right perpendicular subcategory of $\mathbf{K}(A\mbox{-Inj})$ with respect to $DA$, that is,
\[(DA)^\perp = \{X\in \mathbf{K}(A\mbox{-Inj})\mid \mathrm{Hom}_{\mathbf{K}(A\mbox{-}\mathrm{Inj})}(DA, X[n])=0\text{ for any }n\in\mathbb{Z}\}.\]

The following result contains the second part of Main Theorem. Recall that $\Lambda^{-}(Q^{\rm op})=L(Q^{\rm op})^0\ltimes L(Q^{\rm op})^{-1}$ denotes the trivial extension of $L(Q^{\rm op})^0$ by the $L(Q^{\rm op})^0$-bimodule $L(Q^{\rm op})^{-1}$.

\begin{thm}\label{thm:B}
Let $Q$ be a finite quiver without sources. Then there are triangle equivalences
$$\mathbf{K}_{\rm ac}(kQ/J^2\mbox{-}{\rm Proj})\stackrel{\sim}\longrightarrow\mathbf{D}(L(Q^{\rm op})) \stackrel{\sim}\longrightarrow \Lambda^{-}(Q^{\rm op})\mbox{-}\underline{\rm GProj}.$$
Restricting these equivalences to compact objects, we obtain triangle equivalences
$$\mathbf{D}_{\rm sg}(kQ^{\rm op}/J^2)^{\rm op}\stackrel{\sim}\longrightarrow {\rm perf}(L(Q^{\rm op})) \stackrel{\sim}\longrightarrow \Lambda^{-}(Q^{\rm op})\mbox{-}\underline{\rm Gproj},$$
which is further equivalent to  $(L(Q^{\rm op})\mbox{-}{\rm grproj}, (1))$.
\end{thm}

\begin{proof}
Set $A=kQ/J^2$ as above. Recall from Theorem \ref{prop:Koszul} the triangle equivalence $F\colon \mathbf{K}(A\mbox{-}{\rm Inj})\stackrel{\sim}\longrightarrow \mathbf{D}(kQ^{\rm op})$, that sends $DA$ to $\bigoplus_{i\in Q_0} G_i$. Here, we recall that $DA\simeq \bigoplus_{i\in Q_0} I_i$ with $I_i$ the indecomposable injective $A$-module corresponding to $i$, and that the graded simple $kQ^{\rm op}$-module  $G_i$ in given in (\ref{equ:5}). Hence, $F$ identifies $(DA)^\perp$ in $\mathbf{K}(A\mbox{-}{\rm Inj})$ with $(\bigoplus_{i\in Q_0} G_i)^\perp=\{G_i\; |\; i\in Q_0\}^\perp$ in $\mathbf{D}(kQ^{\rm op})$.

Recall from Proposition \ref{prop:Leavittloc}(1) that the homomorphism $\iota_{Q^{\rm op}}\colon kQ^{\rm op}\rightarrow L(Q^{\rm op})$ is a graded universal localization with respect to $\{\xi^{Q^{\rm op}}_i\; |\; i\in Q_0\}$; here, we use the assumption that $Q^{\rm op}$ has no sinks. By Proposition \ref{prop:perpD}, we have a triangle equivalence $\{G_i\; |\; i\in Q_0\}^\perp \stackrel{\sim}\longrightarrow \mathbf{D}(L(Q^{\rm op}))$. In view of the equivalence (\ref{equ:perp}), we have the required equivalence $\mathbf{K}_{\rm ac}(A\mbox{-}{\rm Proj})\stackrel{\sim}\longrightarrow\mathbf{D}(L(Q^{\rm op}))$.

The triangle equivalence $\mathbf{D}(L(Q^{\rm op})) \stackrel{\sim}\longrightarrow \Lambda^{-}(Q^{\rm op})\mbox{-}\underline{\rm GProj}$ follows from Lemma \ref{lem:Leavitt} and Proposition \ref{prop:loc}. In view of the equivalence (\ref{equ:triP}), Proposition \ref{prop:hered} and Lemma \ref{lem:Leavitt}(3), the equivalences on the subcategories of compact objects follow immediately.
\end{proof}

\begin{rem}
We remark that by replacing $Q$ by $Q^{\rm op}$, the equivalences on the categories of compact objects obtained in Theorems \ref{thm:A} and \ref{thm:B} may be identified. To this end, we  recall that for any dg algebra $B$, there is a triangle equivalence \begin{align}\label{equ:rem1}
{\rm perf}(B^{\rm opp})\stackrel{\sim}\longrightarrow {\rm perf}(B)^{\rm op}\end{align}
via the functor ${\rm {\bf R}}{\rm Hom}_{B^{\rm o
pp}}(-, B)$, the right derived functor of ${\rm Hom}_{B^{\rm opp}}(-, B)$. For an algebra $\Lambda$, there is triangle equivalence
\begin{align}\label{equ:rem2}
\Lambda\mbox{-}{\rm \underline{Gproj}} \stackrel{\sim}\longrightarrow (\Lambda^{\rm op}\mbox{-}\underline{{\rm Gproj}})^{\rm op}
\end{align}
induced by ${\rm Hom}_\Lambda(-, \Lambda)$. Recall the isomorphisms $L(Q)^{\rm opp}\simeq L(Q)^{\rm op}$ and $\Lambda^{+}(Q)\simeq \Lambda^{-}(Q)^{\rm op}$ induced by the involution on the Leavitt path algebra.  Then via the equivalences (\ref{equ:rem1}) and (\ref{equ:rem2}) the triangle equivalences on compact objects in Theorems \ref{thm:A} and \ref{thm:B} are identified, up to taking opposite categories.
\end{rem}

We conclude the paper with an immediate consequence of Theorems \ref{thm:A} and \ref{thm:B}.

Recall that two finite-dimensional algebras $A_1$ and $A_2$ are said to be \emph{singularly equivalent}  if there is a triangle equivalence between their singularity categories. By (\ref{equ:tri}) the existence of a triangle equivalence $\mathbf{K}_{\rm ac}(A_1\mbox{-Inj}) \stackrel{\sim}\longrightarrow \mathbf{K}_{\rm ac}(A_2\mbox{-Inj})$ implies that $A_1$ and $A_2$ are singularly equivalent. In general, we do not know whether the converse is true; compare the discussion after \cite[Proposition 2.3]{Kra}. However, the following result shows that the converse holds for certain algebras with radical square zero.

Two graded algebras $B_1$ and $B_2$ are said to be \emph{graded Morita equivalent}  if there is an equivalence $B_1\mbox{-Gr}\stackrel{\sim}\longrightarrow B_2\mbox{-Gr}$ that commutes with the degree-shift functors; they are said to be \emph{derived equivalent} if there exists a triangle equivalence $\mathbf{D}(B_1)\stackrel{\sim}\longrightarrow \mathbf{D}(B_2)$, where graded algebras are viewed as dg algebras with trivial differentials.

\begin{prop}\label{prop:equi}
Let $Q$ and $Q'$ be finite quivers without sinks. Then the following statements are equivalent:
\begin{enumerate}
\item the algebras $kQ/J^2$ and $kQ'/J^2$ are singularly equivalent;
\item the Leavitt path algebras $L(Q)$ and $L(Q')$ are graded Morita equivalent;
\item the Leavitt path algebras $L(Q)$ and $L(Q')$ are derived equivalent;
\item the opposite Leavitt path algebras $L(Q)^{\rm op}$ and $L(Q')^{\rm op}$ are derived equivalent;
\item there is a triangle equivalence $\mathbf{K}_{\rm ac}(kQ/J^2\mbox{-}{\rm Inj}) \stackrel{\sim}\longrightarrow \mathbf{K}_{\rm ac}(kQ'/J^2\mbox{-}{\rm Inj})$;
\item there is a triangle equivalence $\mathbf{K}_{\rm ac}(kQ^{\rm op}/J^2\mbox{-}{\rm Proj}) \stackrel{\sim}\longrightarrow \mathbf{K}_{\rm ac}(kQ'^{\rm op}/J^2\mbox{-}{\rm Proj})$;
\item the algebras $\Lambda^+(Q)$ and $\Lambda^+(Q')$ are Morita equivalent.
\end{enumerate}
\end{prop}

\begin{proof}
We only prove the implications ``$(1)\Rightarrow (2) \Rightarrow (3)\Rightarrow (6)\Rightarrow (1)$", while the proof of ``$(1)\Rightarrow (2) \Rightarrow (4)\Rightarrow (5)\Rightarrow (1)$" is similar. Here, we recall that two graded algebras are graded Morita equivalent if and only if so are their opposite algebras.

For ``$(1)\Rightarrow (2)$",  we observe that the equivalences on compact objects in Theorem \ref{thm:A} imply that there is an equivalence $L(Q)\mbox{-grproj}\stackrel{\sim}\longrightarrow L(Q')\mbox{-grproj}$ that commutes with $(-1)$ and then with $(1)$. This equivalence extends to the whole graded module categories.

The implication ``$(2)\Rightarrow (3)$" follows from a general fact: any two graded Morita equivalent algebras $B_1$ and $B_2$ are derived equivalent. Indeed, from the graded Morita equivalence there is a balanced $B_1$-$B_2$-bimodule $P$ such that as a $B_1$-module $P$ is a projective generator with an isomorphism $B_2^{\rm op}\simeq {\rm End}_{B_1}(P)$ of graded algebras. Then the functor ${\rm Hom}_{B_1}(P, -)\colon \mathbf{D}(B_2)\rightarrow \mathbf{D}(B_1)$ is a triangle equivalence; compare \cite[8.1.4]{Kel98}.

The implication ``$(3)\Rightarrow  (6)$" follows from Theorem \ref{thm:B}, while  ``$(6)\Rightarrow (1)$"  follows from (\ref{equ:triP}).

To complete the proof, it suffices to show the implications ``$(2)\Rightarrow (7)$" and ``$(7) \Rightarrow (5)$". Recall a general fact: if two strongly graded algebra $\Lambda$ and $\Lambda'$ are graded Morita equivalent via a balanced $\Lambda$-$\Lambda'$-bimodule $P$, then the corresponding trivial extension algebras $\Lambda^0\ltimes \Lambda^{1}$ and $\Lambda'^0\ltimes \Lambda'^{1}$ are Morita equivalent via the balanced bimodule $P^0\oplus P^1$. Then the implication ``$(2)\Rightarrow (7)$" follows.

The implication ``$(7) \Rightarrow (5)$" follows from Theorem \ref{thm:A} and the following  general fact: for two Morita equivalent algebras $\Lambda$ and $\Lambda'$, we have a triangle equivalence $\Lambda\mbox{-}\underline{\rm GProj}\stackrel{\sim}\longrightarrow \Lambda'\mbox{-}\underline{\rm GProj}$, that is induced from the Morita equivalence.
\end{proof}

The above result is related to  \cite[Remark 7.3]{Haz12}. Recall that the graded Grothendieck group $K_0^{\rm gr}(B)$ of a graded algebra $B$ is a pre-ordered abelian group which carries an automorphism induced by the degree-shift functor $(1)$ on modules. A graded Morita  equivalence between $L(Q)$ and $L(Q')$ induces an isomorphism $$K_0^{\rm gr}(L(Q))\stackrel{\sim}\longrightarrow K_0^{\rm gr}(L(Q'))$$ of pre-ordered abelian groups with automorphisms. It is conjectured that the converse might be true.

\bibliography{}

\vskip 10pt

 {\footnotesize \noindent Xiao-Wu Chen\\
  School of Mathematical Sciences, University of Science and Technology of
China, Hefei 230026, Anhui, PR China \\
Wu Wen-Tsun Key Laboratory of Mathematics, USTC, Chinese Academy of Sciences, Hefei 230026, Anhui, PR China.\\
URL: http://home.ustc.edu.cn/$^\sim$xwchen\\

\noindent Dong Yang\\
Department of Mathematics, Nanjing University, 22 Hankou Road, Nanjing 210093, P. R. China.
}

\end{document}